\documentclass[11pt,intlimits]{article}
\pdfoutput=1
\usepackage{color}
\usepackage{amsthm,amssymb,amsmath}
\usepackage{amsfonts}
\usepackage{fancyhdr}
\usepackage{graphicx}
\newcommand{\thisday}{June 16, 2009}
\renewcommand{\tilde}{\widetilde}
\newcommand{\Rbb}{\mathbb{R}}
\newcommand{\Nbb}{\mathbb{N}}
\newcommand{\Cbb}{\mathbb{C}}
\newcommand{\Z}{\mathbb{Z}}
\newcommand{\xb}{\mathbf{x}}
\newcommand{\eb}{\mathbf{e}}
\newcommand{\pb}{\mathbf{p}}
\newcommand{\zerob}{\mathbf{0}}
\newcommand{\FC}{\mathcal{F}}
\newcommand{\LC}{\mathcal{L}}
\newcommand{\xib}{\boldsymbol{\xi}}
\newcommand{\omeb}{\boldsymbol{\omega}}
\newcommand{\dr}{\mathrm{d}}
\newcommand{\ir}{\mathrm{i}}
\newcommand{\er}{\mathrm{e}}
\newcommand{\NC}{\mathcal{N}}
\newcommand{\RC}{\mathcal{R}}

\newcommand{\dist}{\operatorname{dist}}
\newcommand{\supp}{\operatorname{supp}}

\newcommand{\chiF}{\widehat{\chi_\Omega}(\xib)}
\newcommand{\vol}{\operatorname{vol}}
\newcommand{\const}{\operatorname{const}}
\newcommand{\re}{\operatorname{Re}}
\newcommand{\im}{\operatorname{Im}}
\newcommand{\talp}{\alpha}
\newcommand{\tilr}{r}
\newcommand{\trm}{\tilr_-}
\newcommand{\eps}{\epsilon}
\theoremstyle{plain}
\newtheorem{thm}{Theorem}[section]
\newtheorem{lem}[thm]{Lemma}
\newtheorem{lemma}[thm]{Lemma}
\newtheorem{cor}[thm]{Corollary}
\newtheorem{conj}[thm]{Conjecture}

\theoremstyle{remark}
\newtheorem{defn}[thm]{Definition}
\newtheorem{exam}[thm]{Example}
\newtheorem{remark}[thm]{Remark}

\numberwithin{equation}{section}

\begin{document}

\title{Fourier transform, null variety, and Laplacian's eigenvalues\thanks{The research has been supported by the Royal Society International collaborative grant with Chile. The work of RB has also been supported by CONICYT/PBCT (Chile) Proyecto Anillo de Investigaci\'on en Ciencia y Tecnolog\'\i a ACT30/2006. The work of LP has been also supported by a Leverhulme Trust grant.}}
\author{
Rafael Benguria\\
\normalsize\small Facultad de F\'{\i}sica\\
\normalsize\small P. Universidad Cat\'{o}lica de Chile\\
\normalsize\small Casilla 306, Santiago 22\\
\normalsize\small Chile\\
\normalsize\small {\sffamily rbenguri@fis.puc.cl}
\and
Michael Levitin\\
\normalsize\small Cardiff School of Mathematics\\ %Wales Institute for Mathematical and Computational Sciences\\
\normalsize\small Cardiff University, and WIMCS\\
\normalsize\small Senghennydd Road, Cardiff CF24 4AG\\
\normalsize\small United Kingdom\\
\normalsize\small {\sffamily Levitin@cardiff.ac.uk}
\and
Leonid Parnovski\\
\normalsize\small Department of Mathematics\\
\normalsize\small University College London\\
\normalsize\small Gower Street, London WC1E 6BT\\
\normalsize\small United Kingdom\\
\normalsize\small {\sffamily leonid@math.ucl.ac.uk}
}
%
%
%\subjclass{}
%
\date{\tiny version 7 \qquad\qquad \thisday}

\maketitle
\begin{abstract} We consider a quantity $\kappa(\Omega)$ --- the distance to the origin from the null variety of the Fourier transform of the characteristic function of $\Omega$. We conjecture, firstly, that $\kappa(\Omega)$ is maximized, among all convex balanced domains $\Omega\subset\Rbb^d$ of a fixed volume,  by a ball, and also that $\kappa(\Omega)$ is bounded above
by the square root of the second Dirichlet eigenvalue of $\Omega$. We prove some weaker versions of these conjectures in dimension two, as well as their validity for domains asymptotically close to a disk, and also discuss further links between $\kappa(\Omega)$ and the eigenvalues of the Laplacians.
\end{abstract}

\

\newpage
%
%%%%%%%%%%%%%%%%%%%%%%%%%%%%%%%%%%%%%%%%%%%%%%%%%%%%%%%%%%%%
\section{Introduction}

Let $\Omega$ be a bounded open domain in $\Rbb^d$ with boundary $\partial\Omega$, let $\xb = (x_1,\ldots,x_d)$ be a vector of Cartesian coordinates in $\Rbb^d$, and let
\[
\chi_{\Omega}(\xb)=\begin{cases}
1,\qquad&\text{if }\xb\in\Omega,\\
0\qquad&\text{if }\xb\not\in\Omega
\end{cases}
\]
denote the characteristic function of $\Omega$.

The complex  Fourier transform of $\chi_{\Omega}(\xb)$,
\[
\chiF=\FC[\chi_{\Omega}](\xib):=
\int_{\Omega} \er^{\ir\xib\cdot\xb}\,\dr\xb
\]
or, more importantly, its \emph{complex null variety}, or \emph{null set},
\[
\NC_\Cbb(\Omega):=\{\xib\in\mathbb{C}^d:\chiF=0\}
\]
has been studied extensively. Particular attention has been attracted by the role it plays in numerous attempts to prove the famous \emph{Pompeiu problem} and \emph{Schiffer's conjecture}. We can refer for example to  \cite{Agr, Avi, Ber, BroKah, BroSchTay, GaSe, Kob1, Kob2}; this list is by no means complete.

Although our paper is not directly related to these still open questions, we recall them as part of the motivation for further study of the null variety.

Let $\mathcal{M}(d)$ be a group of rigid motions of $\Rbb^d$, and $\Omega$ be a bounded simply connected domain with piecewise smooth boundary. The Pompeiu problem is to prove that the existence of a non-zero continuous function $f:\Rbb^d\to\Rbb$ such that
$\int_{\mathbf{m}(\Omega)} f(\xb)\,\dr \xb = 0$ for all $\mathbf{m}\in  \mathcal{M}(d)$ implies that $\Omega$ is a ball.

Schiffer's conjecture is that the existence of an eigenfunction $v$ (corresponding to a non-zero eigenvalue $\mu$) of a Neumann Laplacian on a (simply connected) domain $\Omega$ such that $v\equiv\const$ along the boundary $\partial\Omega$ (or, in other words, the existence of a non-constant solution $v$ to the over-determined problem $-\Delta v=\mu v$, $\partial v/\partial n|_{\partial\Omega}=0$, $v|_{\partial\Omega}=1$) implies that $\Omega$ is a ball.

It is known  that the positive answer to the Pompeiu problem is equivalent to Schiffer's conjecture. Moreover, a domain $\Omega$ would be a counterexample to both if there exists $r>0$ such that
$\NC_\Cbb(\Omega)$ contains the complex sphere $\{\xib\in\Cbb^d: \sum_{j=1}^d \xi_j^2=r^2\}$. One of the common tools in attacking the conjectures has been an asymptotic analysis of the null variety far from the origin in an attempt to prove that such counterexample cannot exist.

In many cases, the study of the null variety in the papers cited above has been restricted to the case of a \emph{convex} domain $\Omega$. Additionally, it is convenient to assume that $\Omega$ is \emph{balanced} (i.e., centrally symmetric with respect to the origin), and to deal instead with the \emph{real null variety}
\[
\NC(\Omega):=\NC_\Cbb(\Omega)\cap\Rbb^d=\{\xib\in\mathbb{R}^d:\chiF=0\}=
\{\xib\in\mathbb{R}^d:\int_\Omega \cos(\xib\cdot\xb)\,\dr\xb=0\}\,.
\]
We assume that $\Omega$ is convex and balanced  in most parts of this paper.

The purpose of this paper is to study the behaviour of the null variety \emph{near} the origin, and its relation with the classical spectral theory. Namely, we define the numbers
\[
\kappa_\Cbb(\Omega) := \dist(\NC_\Cbb(\Omega), \zerob)=\min\{|\xib|:\xib\in\NC_\Cbb(\Omega)\}
\]
and
\[
\kappa(\Omega) := \dist(\NC(\Omega), \zerob)=\min\{|\xib|:\xib\in\NC(\Omega)\}
\]
(if there are no real zeros, we set $\kappa(\Omega)=\infty$). Throughout most of the paper, we will be dealing with the real zeros of Fourier transform and the quantity $\kappa(\Omega)$,  so, unless specified otherwise, we always assume that the argument of the Fourier transform $\FC[\chi_{\Omega}]$ is real.

On the basis of some partial cases presented below in Section~\ref{sec:motivation}, we conjecture that, firstly, $\kappa(\Omega)$ is maximized, among all convex balanced domains of the same volume as $\Omega$, by a ball (see Conjecture  \ref{conj:kap_ball}), and, secondly, that for all convex balanced domains $\kappa(\Omega)$ is bounded above by the square root of the second Dirichlet eigenvalue of $\Omega$ (see Conjecture  \ref{conj:kap_lam}). Note that it is very easy to see that $\kappa(\Omega)$ is always (i.e. without the convexity and central symmetry conditions) bounded below by the square root of the second Neumann eigenvalue of $\Omega$, see Lemma \ref{Neumann}.

Unfortunately we are unable to prove Conjectures  \ref{conj:kap_ball} and \ref{conj:kap_lam} as stated. Even in the planar case $d=2$, when the geometry of convex domains is easier to deal with, we are only able to establish
some weaker versions of these conjectures, see Theorems \ref{thm:kap_ball} and \ref{thm:kap_lam}. However, even these weaker results shed some extra light on the links between $\kappa(\Omega)$ and Dirichlet and Neumann
eigenvalues, and in particular show some surprising links with Friedlander's inequalities between the eigenvalues of these two problems, see Remark \ref{ex:inequalities} and Remark \ref{rem:further_inequalities}. Additionally, we can also establish the validity of Conjectures \ref{conj:kap_ball} and \ref{conj:kap_lam} for small star-shaped  perturbations of a disk, see Theorem \ref{thm:asympt_est}.

We also indicate that our results and conjectures can not be extended to wider classes of domains, in particular when the convexity condition is dropped, see Theorems \ref{thm:counterexample},  \ref{Nazarov}, and Corollary \ref{cor:Nazarov}.

The rest of this paper is organized as follows. Section \ref{sect:statements} contains the statements of our  Conjectures and main Theorems. Some particular cases making the conjectures plausible are collated in Section \ref{sec:motivation}. Some preliminary estimates (which in particular imply the validity of Conjecture \ref{conj:kap_ball} for relatively ``long and thin'' planar convex balanced domains) are presented and proved in
Section \ref{sec:some_estimates}. Extra notation and facts  from convex geometry are in Section \ref{sec:geometric}. Section \ref{sec:main_proof} contains the proof of Theorem \ref{thm:kap_ball}; some auxiliary technical Lemmas used in the proofs are collected in a separate Section \ref{sec:lemmas}. The perturbation-type results are proved in Section~\ref{sec:perturbations}, and the counterexamples for non-convex domains are proved in in Section~\ref{sec:counterex}.

We finish this Section by introducing some additional notation used throughout the paper. We write
$\vol_d(\cdot)$ for a $d$-dimensional Lebesgue measure of a set. Given a unit vector $\eb\in S^{d-1}$, we write
$x_\eb=\xb\cdot\eb$ and $\xb'_\eb=\xb-x_\eb\eb$. We write a real vector $\xib\in\Rbb^d$ in spherical coordinates as  $\xib=(\rho,\omeb)$,  with $\rho=|\xib|$ and $\omeb=\xib/\rho\in S^{d-1}$.  $B_d(R)=\{\xb\in\Rbb^d:|\xb|<R\}$ denotes a ball of radius $R$ centred at $\zerob$, and a shorthand for a unit ball will be $B_d=B_d(1)$.
$\Omega^*$ stands for a  ball in $\Rbb^d$ centred at $\zerob$ and of the same volume as $\Omega$.

Additionally, for  a direction $\eb\in S^{d-1}$, we define
$\kappa_{j}(\eb)=\kappa_j(\eb; \Omega)$ as the $j$-th positive real $\rho$-zero of
$\widehat{\chi_\Omega}(\rho\eb)$ (counting multiplicities); note that
$\displaystyle\NC(\Omega)=\bigcup_{j=1}^\infty\NC_j(\Omega)$, where
\[
\NC_j(\Omega):= \bigcup_{\eb\in S^{d-1}}\kappa_j(\eb; \Omega)\eb\,.
\]

Finally, $J_a(r)$ are the usual Bessel functions of order $a$, and $j_{a,k}$ are their positive zeros numbered in increasing order. The eigenvalues of the Dirichlet Laplacian on $\Omega$ are denoted by $\lambda_k(\Omega)$, $k=1,\dots$, and of the Neumann Laplacian by  $\mu_j(\Omega)$, $j=1,\dots$ ($\mu_1=0$).

\subsection*{Acknowledgments} We would like to express our gratitude to I. Polterovich for stimulating discussions, to F. Nazarov and N. Sidorova for helping with the proof
of Theorem \ref{Nazarov}, and to N.~Filonov for letting us use his Lemma \ref{Neumann2}. We are also grateful to the referee for useful suggestions.

%%%%%%%%%%%%%%%%%%%%%%%%%%%%%%%%%%%%%%%%%%%%%%%%%%%%%%%%%%%%

\section{Conjectures and statements}\label{sect:statements}

\begin{defn}\label{defn:bal}
$\Omega$ is \emph{balanced} if it is invariant with respect to the mapping
$\xb\mapsto -\xb$.
\end{defn}

\begin{conj}\label{conj:kap_ball}
If $\Omega$ is convex and balanced, then
\begin{equation}\label{eq:kap_ball_conj}
\kappa(\Omega)\le\kappa(\Omega^*)\,,
\end{equation}
with the equality iff $\Omega$ is a ball.
\end{conj}

\begin{conj}\label{conj:kap_lam}
If $\Omega$ is convex and balanced, then
\begin{equation}\label{eq:kap_lam_conj}
\kappa(\Omega)\le\sqrt{\lambda_2(\Omega)}\,,
\end{equation}
with the equality iff $\Omega$ is a ball.
\end{conj}

In the next section we consider several explicit examples for which we demonstrate the validity of these conjectures.

Although we believe these Conjectures to be true, we are unable to prove them without some
additional assumptions. We can however establish somewhat weaker forms in the two-dimensional case as stated in the next two theorems. Also, we can prove \eqref{eq:kap_ball_conj} subject to some additional conditions on $\Omega$, see Corollaries~\ref{cor:largeD} and \ref{cor:small_r}, and  Remark \ref{rem:long_lam}.

\begin{thm}\label{thm:kap_ball}
If $d=2$, and $\Omega$ is convex and balanced, then
\begin{equation}\label{eq:kap_ball_proved}
\kappa(\Omega)\le C\kappa(\Omega^*)\,,
\end{equation}
with
\begin{equation}\label{eq:C_js}
C=\tilde{C}:=\frac{2j_{0,1}}{j_{1,1}}\approx 1.2552\,.
\end{equation}
\end{thm}

%\bigskip
%\noindent
% Note that the Example \ref{exam:3.4} below shows that for the two dimensional ball $B$, $\kappa(B)= j_{1,1}$.

\begin{thm}\label{thm:kap_lam}
If $d=2$, and $\Omega$ is convex and balanced, then
\begin{equation}\label{eq:kap_lam_proved}
\kappa(\Omega)\le 2\sqrt{\lambda_1(\Omega)}\,.
\end{equation}
\end{thm}

\begin{remark}
Note that Theorem \ref{thm:kap_lam} immediately follows from Theorem \ref{thm:kap_ball} by the Faber-Krahn inequality,
\[
\lambda_1(\Omega)\ge \lambda_1(\Omega^*)=\frac{\pi j_{0,1}^2}{\vol_2(\Omega)}\,,
\]
and rescaling properties of Lemma \ref{lem:rescaling}. Note also that \eqref{eq:kap_lam_proved} is clearly weaker than  \eqref{eq:kap_lam_conj} in the two-dimensional case, since, by the Payne-P\'{o}lya-Weinberger inequality \cite{PPW}, in two dimensions
\[
\lambda_2(\Omega)<3\lambda_1(\Omega)\,,
\]
or by the even stronger Ashbaugh-Benguria inequality \cite{AshBen},
\[
\lambda_2(\Omega)\le \left(\frac{j_{1,1}}{j_{0,1}}\right)^2 \lambda_1(\Omega)\approx 2.539\lambda_1(\Omega)\,.
\]
Finally, in the one-dimensional case, a convex balanced domain is an interval $(-a, a)=B_1(a)$ for some $a>0$, and
\[
\kappa(B_1(a))=\sqrt{\lambda_2(B_1(a))}=\frac{\pi}{a}\,,
\]
so that  \eqref{eq:kap_ball_conj} and \eqref{eq:kap_lam_conj} hold with equality.
\end{remark}

We can also establish the validity of  \eqref{eq:kap_ball_conj} and \eqref{eq:kap_lam_conj} for balanced star-shaped (but not necessarily convex) domains which are close to a disk. Namely, let
$F:\mathbb{S}^1\to\Rbb$ be a $C^2$ function on the unit circle; we additionally assume that
$F$ is periodic with period $\pi$:
\begin{equation}\label{eq:F_bal}
F(\theta+\pi)=F(\theta)\,.
\end{equation}
For $\eps\ge0$, define a domain in polar coordinates $(r, \theta)$ as
\begin{equation}\label{eq:Omega_F}
\Omega_{\epsilon F}:=\{(r, \theta)\,:\, 0\le r\le 1+\epsilon F(\theta)\}\,.
\end{equation}
Condition \ref{eq:F_bal} implies that $\Omega_{\eps F}$ is balanced.

Assume additionally that $F$ is \emph{area preserving}, that is
\begin{equation}\label{eq:intF0}
\int_0^{2\pi} F(\theta)\,\dr \theta = 0\,,
\end{equation}
and so
\[
\vol_2(\Omega_{\epsilon F})=\pi+O(\eps^2)\,.
\]
As we shall see from the re-scaling properties summarized in Lemma \ref{lem:rescaling},
condition \eqref{eq:intF0} can be assumed without any loss of generality.

The unperturbed domain (when $\eps=0$), $\Omega_{0F}$, is just a unit planar disk $B_2$.

We have

\begin{thm}\label{thm:asympt_est}
Let us fix a non-zero function $F$ as above satisfying  \eqref{eq:F_bal} and \eqref{eq:intF0}. Then the one-sided derivatives satisfy
\begin{equation}
\left.\frac{\dr\kappa(\Omega_{\eps F})}{\dr \eps}\right|_{\eps=0+}<0\,,
\end{equation}
and
\begin{equation}\label{eq:former2.10}
\left.\frac{\dr\kappa(\Omega_{\eps F})}{\dr \eps}\right|_{\eps=0+}< \left.\frac{\dr\sqrt{\lambda_2(\Omega_{\eps F})}}{\dr \eps}\right|_{\eps=0+}\,.
\end{equation}
Consequently, for sufficiently small $\eps>0$ (depending on $F$), Conjectures \ref{conj:kap_ball} and \ref{conj:kap_lam} with $\Omega=\Omega_{\eps F}$ hold.
\end{thm}

On the other hand, there exist arbitrarily small star-shaped non-convex perturbations of the disk for which at least \eqref{eq:kap_ball_conj} does not hold. Namely, we have
\begin{thm} \label{thm:counterexample}
For each positive $\tilde\delta$, there exists a balanced star-shaped domain $\Omega$ with $\vol_{2}(\Omega)=\pi$ and such that $B(0,1-\tilde\delta)\subset\Omega\subset B(0,1+\tilde\delta)$, for which $\kappa(\Omega)>j_{1,1}$ .
\end{thm}

Continuing formulating negative results, we have the following

\begin{thm}\label{Nazarov}
There is no $C$ such that \eqref{eq:kap_ball_proved} holds uniformly for all (not necessarily connected) balanced one-dimensional domains $\Omega$.
\end{thm}

From this, we immediately have

\begin{cor}\label{cor:Nazarov}
There is no $C$ such that \eqref{eq:kap_ball_proved} holds uniformly for all balanced  connected two-dimensional domains $\Omega$.
\end{cor}

Theorem \ref{thm:counterexample} and Corollary \ref{cor:Nazarov} show that convexity plays a crucial role in Theorem \ref{thm:kap_ball}
and Conjecture \ref{conj:kap_ball}

%%%%%%%%%%%%%%%%%%%%%%%%%%%%%%%%%%%%%%%%%%%%%%%%%%%%%%%%%%%%

\section{Motivation and elementary domains}\label{sec:motivation}

We start with two trivial results, which are immediate by the change of variables, and which in particular show that  our conjectures are scale invariant. Let $\RC_\alpha$ denotes a mapping $(x_1,x_2,\dots,x_d)\mapsto (\alpha x_1,x_2,\dots,x_d)$, $\alpha>0$.

\begin{lem}\label{lem:rescaling1d}
For any $\Omega\subset\Rbb^d$, 
\[
\NC(\RC_\alpha\Omega)=\RC_{1/\alpha}\NC(\Omega)\,.
\]
\end{lem}

\begin{lem}\label{lem:rescaling}
Let $\Omega'$ be the image of $\Omega\subset\Rbb^d$ under a homothety with coefficient $\alpha>0$. Then
\[
\kappa(\Omega')=\frac{1}{\alpha}\kappa(\Omega)\,,\qquad
\lambda_j(\Omega')=\frac{1}{\alpha^2}\lambda_j(\Omega)\,.
\]
\end{lem}

The following result illustrates that there exists a relation between the null variety and eigenvalues of the
Neumann Laplacian, which makes Conjecture \ref{conj:kap_lam} even more intriguing.

\begin{lem}\label{Neumann}
For any   $\Omega\subset\Rbb^d$,
\begin{equation}
\kappa(\Omega)\ge \kappa_\Cbb(\Omega)\ge\sqrt{\mu_2(\Omega)}\,.
\end{equation}
\end{lem}

\begin{proof} Let $\xib_0\in\NC_\Cbb(\Omega)$, and so $\int_{\Omega} \er^{\ir\xib_0\cdot\xb}\,\dr\xb=0$. This means
that $\langle \er^{\ir\xib_0\cdot\xb},1\rangle_{L_2(\Omega)}=0$, so that $\phi:=\er^{\ir\xib_0\cdot\xb}$ is a test
function for $\mu_2(\Omega)$ (obviously, $\phi\in H^1(\Omega)$). But, by direct computation,
\[
\frac{\|\nabla\phi\|^2_{L_2(\Omega)}}{\|\phi\|^2_{L_2(\Omega)}}=|\xib_0|^2\,.
\]
Thus, $|\xib_0|^2\ge \mu_2(\Omega)$ for any $\xib_0\in\NC_\Cbb(\Omega)$, whence the result.
\end{proof}

In fact, as was shown to us by N. Filonov \cite{Fil2}, one can improve this result to obtain

\begin{lem}\label{Neumann2}
For any   $\Omega\subset\Rbb^d$,
\begin{equation}
\kappa(\Omega)\ge 2\sqrt{\mu_2(\Omega)}\,.
\end{equation}
\end{lem}

\begin{proof} By the variational principle,
\[
\mu_2(\Omega)\le\sup_{\phi\in\mathcal{L}_2}  \frac{\|\nabla\phi\|^2_{L_2(\Omega)}}{\|\phi\|^2_{L_2(\Omega)}}
\]
for any linear subspace $\mathcal{L}_2\subset H^1(\Omega)$ such that $\dim\mathcal{L}_2=2$. Choose
$\xib_0\in\NC(\Omega)$, and set $\LC_2=\operatorname{span}(\er^{\ir\xib_0\cdot\xb/2}, \er^{-\ir\xib_0\cdot\xb/2})$. 
The elements of $\LC$ are linearly independent, and the result immediately follows by direct computation.
\end{proof}

\begin{exam}[\bfseries A ball in $\Rbb^d$]\label{exam:3.4} For a unit ball $B_d$ and real $\xib$, we have:
\begin{equation}\label{eq:kappa_for_ball}
\widehat{\chi_{B_d}}(\xib)=(2\pi)^{d/2} \frac{J_{d/2}(|\xib|)}{|\xib|^{d/2}}\,,
\end{equation}
and so
\begin{equation}
\kappa(B_d)=j_{d/2,1}\,.
\end{equation}

On the other hand,
\[
\lambda_2(B_d)=\lambda_3(B_d)=\dots=\lambda_{1+d}(B_d)=j_{d/2,1}^2=(\kappa(B_d))^2\,.
\]
\end{exam}

For illustration, we give a proof of \eqref{eq:kappa_for_ball} in dimension $d=2$. We choose the direction of $\xib$ as the $x_1$-axis, and write, in polar coordinates, $\xb = (r\cos\theta, r\sin\theta)$. Thus,
\[
\widehat{\chi_{B_d}}(\xib)=\int_0^1 \int_0^{2\pi} e^{\ir |\xib| r \cos \theta}r \,  \dr r \dr\theta\,.
\]
Then we use formula \cite[formula 9.1.18]{AbrSte}, i.e.,
\[
J_0(z)=\frac{1}{2\pi} \int_0^{2\pi} \cos(z \cos \theta) \, d \theta,
\]
to express the previous integral as
\[
\widehat{\chi_{B_d}}(\xib)= 2 \pi \int_0^1 J_0(|\xib|r)r  \, \dr r=\frac{2\pi}{|\xib|^2} \int_0^{|\xib|}
J_0(s) s \, \dr s\,.
\]
Finally, we use the raising and lowering relations for Bessel
functions embodied in  \cite[formula 9.1.27]{AbrSte} (third formula with
$\nu=1$), i.e.,
\[
J_1'(r)=J_0(r)-\frac{1}{r} J_1(r)\,,
\]
which can be expressed in the more convenient form,
\[
rJ_0(r) =(rJ_1(r))'\,.
\]
Thus, we get,
\[
\widehat{\chi_{B_d}}(\xib)=\frac{2\pi}{|\xib|^2} \int_0^{|\xib|} (sJ_1(s))'\, \dr s = \frac{2\pi}{|\xib|}
J_1(|\xib|),
\]
which is the desired equality \eqref{eq:kappa_for_ball}  in two dimensions. The corresponding
formula in any dimension is equally simple to establish.

\begin{exam}[\bfseries A cuboid in $\Rbb^d$]\label{exam:cube} Consider, for $d\ge 2$, a cuboid $P$ with edge lengths
$a_1\ge a_2\ge\dots\ge a_d>0$.
We have:
\begin{equation}\label{eq:lam2_P}
\lambda_2(P) = \pi^2 \left(4a_1^{-2} + \sum_{j=2}^d (a_j)^{-2}\right).
\end{equation}
On the other hand, if $\xib\in\NC(P)$, we have $\prod_{j=1}^d \sin(\xi_j a_j/2)=0$, and so
$|\xib|$ is minimized by the vector $(2\pi/a_1,0,\dots,0)$, giving
\begin{equation}\label{eq:kap_P}
\kappa(P) = \frac{2\pi}{a_1}<\sqrt{\lambda_2(P)}.
\end{equation}

Proving \eqref{eq:kap_ball_conj} for $P$ requires a bit more effort. We have
\[
P^*=B_d(R)\quad\text{with }
R=\frac{\left(a_1\cdot a_2\cdots a_d\cdot\Gamma(1+d/2)\right)^{1/d}}{\sqrt{\pi}}\,,
\]
and, after some transformations, the required inequality is reduced to
\[
j_{d/2, 1}\ge 2\sqrt{\pi}\left(\Gamma\left(1+\frac{d}{2}\right)\right)^{1/d}\,.
\]
This, in turn, is proved using a combination of Stirling's formula, Lorch's lower bound
$j_{\nu,1}\ge\sqrt{(\nu+1)(\nu+5)}$ \cite{Lor}, and numerical checks for low $d$.
\end{exam}

\begin{exam}[\bfseries A right-angled triangle in $\Rbb^2$] Let $T=T_{1,a}$ be a right-angled triangle with sides $1$,
$a>1$, and $\sqrt{1+a^2}$. One can check, after some computations, that
\[
\kappa(T_{1,a})=2\pi\sqrt{1+a^{-2}}\,.
\]

We remark that both inequalities \eqref{eq:kap_ball_conj} and \eqref{eq:kap_lam_conj} with $\Omega = T$ hold for values of $a$ sufficiently close to one,  but fail for large $a$  or small $a$. This can be checked either by direct computation  (in case of  \eqref{eq:kap_ball_conj}) or by domain monotonicity  (in case of  \eqref{eq:kap_lam_conj}),
by comparing $\lambda_2(T)$ with either $\lambda_2(T_{a,a})=10\pi^2/a^2$ (for small $a$) or
with the second eigenvalue of the rectangle with sides $3/4$ and $a/4$ (for large $a$).

Note that $T$ is not balanced and  we do not conjecture that \eqref{eq:kap_ball_conj} and \eqref{eq:kap_lam_conj}
hold  in general for such domains. It may be plausible that $\kappa_\Cbb(\Omega)\le\kappa(\Omega^*)$ and $\kappa_\Cbb(\Omega)\le\sqrt{\lambda_2(\Omega)}$ for general convex domains, however the study of complex null varieties is outside the scope of this paper.
\end{exam}

\begin{exam}[\bfseries Numerics]
We have also verified  Conjectures~ \ref{conj:kap_ball} and \ref{conj:kap_lam} numerically.  We have conducted (jointly with Brian Krushave, an undergraduate student at Heriot-Watt University, whose research was funded by a Nuffield Foundation undergraduate bursary) a large number of calculations for different multiparametric families of balanced convex domains in the two-dimensional case. A typical example would be a family of rectangles with different circular or elliptic segments added along their sides, in order to produce some stadium-like domains.

The zeros of Fourier transform were found by analytic or numerical integration and minimization, and the eigenvalues of the Dirichlet Laplacian by the finite element method.
\end{exam}

\begin{remark}[\bfseries Estimates of the spectrum]\label{ex:inequalities}

We would like to show how to use estimates of $\kappa(\Omega)$ in spectral inequalities between the eigenvalues $\lambda_n=\lambda_n(\Omega)$ of the Dirichlet Laplacian on $\Omega$ and the eigenvalues $\mu_n=\mu_n(\Omega)$ of the Neumann Laplacian on the same domain. It is known that for general domains we have
\begin{equation}\label{eq:mulam1}
\mu_{n+1}<\lambda_n
\end{equation}
 for each $n$, and, moreover, for convex domains in $\Rbb^d$
we have  
\begin{equation}\label{eq:mulam2}
\mu_{n+d}<\lambda_n
\end{equation}
 \cite{LevWei}. It was conjectured that \eqref{eq:mulam2} holds for all domains; this conjecture remains open, and we remark that 
 a `counterexample' given in the paper by Levine and Weinberger is erroneous.
 
 Estimate \eqref{eq:mulam1} was
proved by Friedlander \cite{Fri} for domains with smooth boundaries; later,
an elegant proof for arbitrary domains was obtained by Filonov \cite{Fil}.
Filonov's proof goes like this. Let $n$ be fixed. Denote
by $\phi_j$ the Dirichlet eigenfunctions of $\Omega$. By the min-max principle,
in order to prove $\mu_{n+1}\le\lambda_n$, it is enough to find a
subspace $\LC$ of $H^1(\Omega)$ such that $\dim \LC=n+1$ and for each $\phi\in \LC\setminus\{0\}$ we have
\begin{equation}\label{lf1}
{\int_{\Omega}|\nabla\phi|^2\,\dr\xb}\le\lambda_n{\int_{\Omega}\phi^2\,\dr\xb}.
\end{equation}
Put $\LC=\operatorname{span}(\phi_1(\xb),\dots,\phi_n(\xb),e^{i\xib\cdot\xb})$, where $\xib$ is any
real vector satisfying $|\xib|^2=\lambda_n$. Obviously, $\dim \LC=n+1$. Suppose now that
$\phi\in \LC$. This means that
\[
\phi=\sum_{j=1}^n a_j\phi_j+b\er^{\ir\xib\cdot\xb}.
\]
Then the left-hand side of of \eqref{lf1} is
\[
\sum_{j=1}^n |a_j|^2\lambda_j+|b|^2\lambda_n\vol_d(\Omega)+
2\sum_{j=1}^n\re\left(a_jb \lambda_n\int_{\Omega}\phi_j\er^{\ir\xib\cdot\xb}\,\dr\xb\right)
\]
(in the last sum, we have integrated by parts using the fact that $\phi_j$ satisfies Dirichlet boundary
conditions on $\partial\Omega$ and that $|\xib|^2=\lambda_n$).
The right-hand side of \eqref{lf1} is
\[
\lambda_n\left(\sum_{j=1}^n |a_j|^2+|b|^2\vol_d(\Omega)
+2\sum_{j=1}^n\re\left(a_jb\int_{\Omega}\phi_j\er^{\ir\xib\cdot\xb}\,\dr\xb\right)\right).
\]
Comparing the last two expressions leads to \eqref{lf1}.

Now suppose we want to improve this result and to show \eqref{eq:mulam2} that for some class of (not necessarily convex) domains $\mu_{n+2}\le\lambda_n$. The natural
approach to try is to add one more exponential to $L$, namely to put
\[
\LC=\operatorname{span}(\phi_1(\xb),\dots,\phi_n(\xb),\er^{\ir\xib_1\cdot\xb},\er^{\ir\xib_2\cdot\xb}),
\]
\begin{equation}\label{lf2}
|\xib_j|^2=\lambda_n\,,\qquad j=1,2.
\end{equation}
Then, in order
for \eqref{lf1} to hold, we must get rid of the cross-term with two exponentials, i.e. we must assume that
\[
\int_{\Omega}\er^{\ir(\xib_1-\xib_2)\cdot\xb}\,\dr\xb=0.
\]
In the notation introduced above, this means that
\begin{equation}\label{lf3}
\xib_1-\xib_2\in\NC(\Omega).
\end{equation}
Obviously, we can choose vectors $\xib_1$, $\xib_2$ satisfying both \eqref{lf2} and \eqref{lf3} iff
$\kappa(\Omega)\le 2\sqrt{\lambda_n(\Omega)}$. Thus, if we could show that for some, not necessarily convex,  $d$-dimensional domain $\Omega$, the estimate \eqref{eq:kap_lam_proved} holds, then the inequality $\mu_{n+2}(\Omega)\le\lambda_n(\Omega)$ will hold for each $n$. Similarly, for any $\Omega$, if we know a number $n_0$ such that $\kappa(\Omega)\le 2\sqrt{\lambda_{n_0}(\Omega)}$, then the inequality $\mu_{n+2}(\Omega)\le\lambda_n(\Omega)$ is guaranteed to hold for $n\ge n_0$.
\end{remark}

%%%%%%%%%%%%%%%%%%%%%%%%%%%%%%%%%%%%%%%%%%%%%%%%%%%%%%%%%%%%

\section{Some estimates of $\kappa(\Omega)$ for convex balanced domains}\label{sec:some_estimates}

Throughout this section $\Omega$ is convex and balanced,  and $\Omega$ dependence is frequently dropped; also we always work with \emph{real} zeros of the Fourier transform. 

Our aim here is to prove the following

\begin{thm}\label{thm:diamest}
Suppose that $d=2$ and $D(\Omega)$ is the diameter of $\Omega$. Then
\begin{equation}\label{eq:ineqdiam}
\kappa(\Omega)\le \frac{4\pi}{D(\Omega)}\,.
\end{equation}
\end{thm}

\begin{remark}
After this paper was written, we have discovered that Theorem \ref{thm:diamest} had been previously proved in \cite{Zas}. 
\end{remark}

Note that Theorem \ref{thm:diamest} immediately implies
\begin{cor}\label{cor:largeD}
Conjecture \ref{conj:kap_ball} holds for  convex, balanced domains $\Omega\subset\Rbb^2$ such that the diameter $D(\Omega)$
satisfies
\begin{equation}\label{eq:largeD}
\frac{\sqrt{\pi}D(\Omega)}{2\sqrt{\vol_2(\Omega)}}\ge\frac{2\pi}{j_{1,1}}\,.
\end{equation}
\end{cor}

The scaling in \eqref{eq:largeD} is chosen in such a way that its left--hand side equals one for a disk.

Let $r_-(\Omega)$ be the \emph{inradius} of a convex balanced domain $\Omega$. Then it is easy to see that there exists a rectangle with sides $2r_-(\Omega)$ and $D(\Omega)$ which contains $\Omega$. Thus
$2r_-(\Omega)D(\Omega)\ge \vol_2(\Omega)$, which together with Corollary \ref{cor:largeD} immediately implies

 \begin{cor}\label{cor:small_r}
 Conjecture \ref{conj:kap_ball} holds for  convex, balanced domains $\Omega\subset\Rbb^2$ such that
 inradius $r_-(\Omega)$ satisfies
\begin{equation}\label{eq:small_r}
\frac{\sqrt{\pi}r_-(\Omega)}{\sqrt{\vol_2(\Omega)}}\le\frac{j_{1,1}}{8}\,.
\end{equation}
\end{cor}

\begin{remark}\label{rem:long_lam}
In the same spirit, one can also establish the validity of Conjecture \ref{conj:kap_lam} subject to additional geometric constraints:
if a domain is sufficiently ``long'' (i.e. the left-hand side of
 \eqref{eq:largeD} is sufficiently large or the left-hand side of \eqref{eq:small_r} is sufficiently small), then \eqref{eq:kap_lam_conj} holds. However such an estimate would be  non-explicit, as there is no explicit isoperimetric bound on the second Dirichlet eigenvalue for convex domains, see \cite{Hen}.
\end{remark}

Before proving Theorem \ref{thm:diamest}, we need to introduce some auxiliary notation, and establish some technical facts.

Fix $\eb\in S^{d-1}$, and define the function
$\nu_\eb:\Rbb\to \Rbb$ by
\[
\nu_\eb(t) =\vol_{d-1}\left(\{\xb: x_\eb=t\}\cap\Omega\right)
\]

It is easy to see that $\nu_\eb$ is an even function and has a compact support $\supp \nu_\eb=[-w(\eb),w(\eb)]$, where
$w$ is the support function of $\Omega$, i.e. $w(\eb)$ is a half-breadth of $\Omega$ in direction $\eb$. If $\Omega$ is convex and
$d=2$, then $\nu_\eb$ is a concave function on $[-w(\eb),w(\eb)]$ (this is not true if $d\ge 3$, e.g. when
$\Omega$ is a cube and $\eb$ is a diagonal but in general Brunn--Minkowski inequality implies that
$(\nu_\eb(\rho))^{1/(d-1)}$ is concave on $[-w(\eb),w(\eb)]$). Thus, $\nu_\eb(t)\le\nu_\eb(0)$ and the function $\nu_\eb$ is non-increasing on
$[0, w(\eb)]$.

As we are working with real zeros of the Fourier transform, we can instead work with
\[
\widehat{\chi}_\eb(\rho):=\widehat{\chi}(\rho\eb)=\int_\Omega \cos(\rho\eb\cdot \xb)\,\dr\xb =
2\int_{0}^{w(\eb)} \cos(t\rho) \nu_{\eb}(t)\,\dr t\,.
\]

\begin{lem}\label{lem:convineqs}
Let $Z:[0,z]\to\Rbb$ be %smooth, %non-negative,
non-increasing and concave. %(i.e. $Z(t)\ge0$, $Z'(t)\le0$, and $Z''(t)\le 0$ for $t\in(0,z)$).
Then
%%%
\begin{align}
\int_{2\pi k}^{2\pi (k+1)} Z(t)\cos(t)\dr t &\le 0\,,\label{eq:convineq1}\\
\int_{2\pi (k+1/2)}^{2\pi (k+3/2)} Z(t)\cos(t)\dr t &\ge 0\,,\label{eq:convineq2}\\
\int_{2\pi k}^{2\pi (k+1/2)} Z(t)\cos(t)\dr t &\ge 0\,,\label{eq:convineq3}\\
\int_{2\pi (k+1/2)}^{2\pi (k+1)} Z(t)\cos(t)\dr t &\le 0\label{eq:convineq4}
\end{align}
%%%%
for $k\in\Nbb$ (assuming that all intervals of integration are inside $[0,z]$).
\end{lem}

\begin{proof} Let $L(t)$ be a linear function such that $L(2\pi(k+1/4))=Z(2\pi(k+1/4))$ and
$L(2\pi(k+3/4))=Z(2\pi(k+3/4))$. Then, by concavity of $Z(t)$, we have
$Z(t)\ge L(t)$ for $t\in [2\pi(k+1/4), 2\pi(k+3/4)]$ (note that $\cos(t)\le 0$ for these values of $t$)
and also $Z(t)\le L(t)$ for $t\in [2\pi k, 2\pi(k+1/4)]\cup[2\pi(k+3/4),2\pi(k+1)]$
(note that $\cos(t)\ge 0$ for these values of $t$). Therefore,
\[
\int_{2\pi k}^{2\pi (k+1)} Z(t)\cos(t)\dr t \le \int_{2\pi k}^{2\pi (k+1)} L(t)\cos(t)\dr t =0\,,
\]
the last equality easily checked by a direct computation. This proves
\eqref{eq:convineq1}, and \eqref{eq:convineq2} is being dealt with similarly.

Further,
\[
\int_{2\pi k}^{2\pi (k+1/2)} Z(t)\cos(t)\dr t=
\int_{2\pi k}^{2\pi (k+1/4)} (Z(t)-Z(2\pi(k+1/2)-t))\cos(t)\dr t\ge 0\,,
\]
since the integrand is non-negative. Inequality  \eqref{eq:convineq4} is similar.
\end{proof}

\begin{remark} Note that \eqref{eq:convineq1} and \eqref{eq:convineq2} require only
concavity of a function $Z$, whereas \eqref{eq:convineq3} and \eqref{eq:convineq4} require only
its monotonicity.
\end{remark}

Recall that $\kappa_j(\eb)$ denotes the $j$-th $\rho$-root (counted  in increasing order with account of multiplicities) of $\widehat{\chi}_\eb(\rho)$.

\begin{lem}\label{lem:kappaj}
Let $d=2$, then
\[
\kappa_{j}(\eb)\le \frac{\pi(j+1)}{w(\eb)}\,.
\]
\end{lem}

\begin{proof} We have $\widehat{\chi}_\eb(0)>0$. Set $\rho_j=\frac{j\pi}{w(\eb)}$.
Let us show that
\[
\widehat{\chi}_\eb(\rho_j)=\widehat{\chi}_\eb\left(\frac{j\pi}{w(\eb)}\right)
\]
is non-negative when
 $j$ is odd, and is non-positive when $j$ is even.

Assume $j=2k$. Then
\[
\begin{split}
\widehat{\chi}_\eb(\rho_j)&=2\int_{0}^{w} \cos\left(\frac{2\pi k t}{w}\right) \nu_\eb(t)\,\dr t=
\frac{w}{\pi k} \int_{0}^{2\pi k} \cos(\tau)\nu_\eb\left(\frac{\tau w}{2\pi k}\right)\,\dr\tau\\
&=\frac{w}{\pi k}\sum_{\ell=0}^{k-1}   \int_{2\pi\ell}^{2\pi(\ell+1)} \cos(\tau)\nu_\eb\left(\frac{\tau w}{2\pi k}\right)\,\dr\tau\,,
\end{split}
\]
which is non-positive by \eqref{eq:convineq1}.

Assume now that $j=2k+1$. Then
\[
\begin{split}
\widehat{\chi}_\eb(\rho_j)&=2\int_{0}^{w} \cos\left(\frac{2\pi (k+1) t}{w}\right) \nu_\eb(t)\,\dr t=
\frac{w}{\pi k} \int_{0}^{2\pi (k+1)} \cos(\tau)\nu_\eb\left(\frac{\tau w}{2\pi (k+1)}\right)\,\dr\tau\\
&=\frac{w}{\pi k} \int_{0}^{\pi} \cos(\tau)\nu_\eb\left(\frac{\tau w}{2\pi (k+1)}\right)\,\dr\tau\\
&\quad+\frac{w}{\pi k}\sum_{\ell=0}^{k-1}   \int_{2\pi(\ell+1/2)}^{2\pi(\ell+3/2)} \cos(\tau)\nu_\eb\left(\frac{\tau w}{2\pi (k+1)}\right)\,\dr\tau\,,
\end{split}
\]
which is non-negative by \eqref{eq:convineq2} and \eqref{eq:convineq3}.

The result now follows from the Intermediate Value theorem (if, for example, $\widehat{\chi}_\eb(\rho)$ is positive except at the points
$\rho_{2k}$ where it is zero, then each point $\rho_{2k}$ is a zero of multiplicity (at least) two, so we still have
$\kappa_{j}(\eb)\le \frac{\pi(j+1)}{w(\eb)}$)
\end{proof}

Lemma \ref{lem:kappaj} immediately leads  to the main result of this section.

\begin{proof}[Proof of Theorem~\ref{thm:diamest}] By Lemma \ref{lem:kappaj},
\[
\kappa(\Omega)=\inf_{\eb\in S^1} \kappa_1(\eb)\le \inf_{\eb\in S^1}\frac{2\pi}{w(\eb)}=
\frac{2\pi}{\sup_{\eb\in S^1} w(\eb)}=\frac{4\pi}{D(\Omega)}\,.
\]
\end{proof}

\begin{remark}\label{rem:further_inequalities}
It was proved in \cite{Zas} that the function $\kappa_1(\eb)$ is  continuous. Using this fact, one can
establish further relationship between this function and Neumann eigenvalues, similar to Lemma \ref{Neumann}.
For example,  we have:
\begin{equation}
\max_{\eb\in S^1}\kappa_1(\eb)\ge \sqrt{\mu_3(\Omega)}.
\end{equation}
Indeed, recall that  $\NC_1=\NC_1(\Omega)=\displaystyle\bigcup_{\eb\in S^1}\kappa_1(\eb)\eb$. Obviously, $\NC_1(\Omega)\subset\NC(\Omega)$. Assuming the
continuity of $\kappa_1(\eb)$, we see that $\NC_1$ is a continuous closed curve having the origin inside it. Let
$\eb_0$ be arbitrary unit vector so that $\pb_0:=\kappa_1(\eb_0)\eb_0\in\NC_1$. Then the closed curve $\pb_0+\NC_1$
obviously contains both the points inside $\NC_1$ (the origin) and outside $\NC_1$ (for example, the point $2\pb_0$). Therefore, the intersection
$(\pb_0+\NC_1)\cap\NC_1$ is non-empty, say $\pb_1\in (\pb_0+\NC_1)\cap\NC_1$. Then three points $\pb_0$, $\pb_1$, and
$\pb_0-\pb_1$ all belong to $\NC_1$. Now we can argue as in the proof of Lemma \ref{Neumann}, with $\er^{\ir\pb_0\cdot\xb}$,
$\er^{\ir\pb_1\cdot\xb}$ and $1$ being three mutually orthogonal test-functions. This shows that $\sqrt{\mu_3(\Omega)}\le
\max(|\pb_1|,|\pb_2|)\le\max\limits_{\eb\in S^1}\kappa_1(\eb)$.
\end{remark}

\begin{remark}\label{rem:worse_contants}
Using the results of this section and the fact that $D(\Omega) \ge 2\sqrt{{\rm vol}_2(\Omega)/\pi}$, we obtain
\[
\kappa(\Omega)\le \frac{4\pi}{D(\Omega)}\le\frac{2\pi^{3/2}}{\sqrt{{\rm vol}_2(\Omega)}}=\frac{2\pi}{j_{1,1}}\kappa(\Omega^*)\,,
\]
thus proving \eqref{eq:kap_ball_proved} with a numerical constant
\[
C=\tilde{C}_1=\frac{2\pi}{j_{1,1}}\approx 1.6398\,.
\]
\end{remark}

\begin{remark}\label{rem:3d_no_diam}
It should be noted that there is no analog of Theorem \ref{thm:diamest} in dimensions higher than two, i.e. one cannot estimate $\kappa(\Omega)$ in terms of the diameter $D(\Omega)$. Indeed, let $S=\{(x_1, x_2)\in\Rbb^2: |x_1|+|x_2|<1\}$, and let $T$ be the three-dimensional body of revolution obtained by rotating $S$ around the $x_1$-axis. Also, choose $\alpha>0$, and set
$T_\alpha=\{\xb\in\Rbb^3: (x_1, \alpha x_2, \alpha x_3)\in T\}$. Then, as $\alpha\to\infty$,  the distances to origin of all zeros of $\widehat{\chi_{T_\alpha}}$, which are not proportional to $\eb_1=(1,0,0)$, tend to $\infty$  by Lemma \ref{lem:rescaling1d}. On the other hand,
\[
\widehat{\chi_{T_\alpha}}(\xi\eb_1)=\alpha^{-2}\widehat{\chi_{T}}(\xi\eb_1)=2\pi\alpha^{-2}\int_0^1 (1-x)^2\cos(x\xi)\dr x=\frac{2\pi\alpha^{-2}}{\xi^3}(\xi-\sin\xi)>0
\]
for all $\xi\in\Rbb$. Thus, $\kappa(T_\alpha)\to\infty$ as $\alpha\to\infty$, while $D(T_\alpha)=2$. A similar example works in any higher dimension.
\end{remark}

%%%%%%%%%%%%%%%%%%%%%%%%%%%%%%%%%%%%%%%%%%%%%%%%%%%%%%%%%%%%

\section{Geometric notation for planar balanced star-shaped domains}\label{sec:geometric}

We set, for a balanced star-shaped domain $\Omega\subset\Rbb^2$, and $r\ge 0$,

\[
\eta(r; \Omega):=\vol_{1}(\Omega\cap\{|\xb|=r\})
\]
and
\[
\alpha(r; \Omega):=\frac{1}{\vol_2 (\Omega)}\int_0^r \eta(\rho; \Omega)\,\dr\rho = \frac{\vol_2 (\Omega\cap B_2(r))}{\vol_2 (\Omega)}
\]
(the normalizing factor $1/\vol_2(\Omega)$ will simplify the computations later on).

Let us also define the numbers
\[
r_-=r_-(\Omega)=\min_{\eb\in S^1} w(\eb)\,,\qquad  r_+:= \max_{\eb\in S^1} w(\eb)\,.
\]
Obviously, $r_-$ is the inradius of $\Omega$ and $2r_+$ is its diameter.

Some properties of the functions $\eta$ and $\alpha$ and the numbers $r_\pm$ are obvious:
\begin{itemize}
\item Both $\eta(r)$ and $\alpha(r)$ are non-negative; additionally, $\alpha(r)$ is non-decreasing;
\item $\eta(r)\equiv 2\pi r$  and $\alpha(r)\equiv \pi r^2/\vol_2(\Omega)$  for $r\le r_-$; moreover,
$r_-=\sup\{r\,:\,\eta(r)=2\pi r\}=\sup\{r\,:\,\alpha(r)=\pi r^2/\vol_2(\Omega)\}$;
\item $\eta(r)\equiv 0$  and $\alpha(r)\equiv \const=1$  for $r\ge r+$; moreover,
$r_+=D/2=\inf\{r\,:\,\eta(r)=0\}=\inf\{r\,:\,\alpha(r)=1\}$ and $\supp \eta=[0, D/2]$.
\end{itemize}

An additional important property is valid for convex domains.
\begin{lem} Let $\Omega\subset\Rbb^2$ be a balanced convex domain.
Then for $r\in [r_-(\Omega), r_+(\Omega)]$, the function $\eta(r)$ is decreasing and the function
$\alpha(r)$ is concave.
\end{lem}

\begin{proof}
Let us prove that the function $\eta$ is decreasing in the given interval.
Indeed, suppose $r_-<r_1<r_2$. Since $\eta(r_1)<2\pi r_1$, we have:
\[
\left(\Omega\cap\{|\xb|=r_1\}\right)\ne \{|\xb|=r_1\}.
\]
Thus, the set $G:=\Omega\cap\{|\xb|=r\}$ consists of several (possibly, infinitely many,
but at least two) circular arcs, say $G_1,\dots,G_n,\dots$.
Note that $G$ is obviously symmetric with respect to the origin, so if $G_j$ is one of the arcs
of $G$, then the symmetric arc, $\tilde G_j$ is also a part of $G$. Let $S_j$ be the strip based on
$G_j$ and $\tilde G_j$ (i.e. $S_j$ is the smallest centrally symmetric strip containing $G_j$ and
$\tilde G_j$, see Figure \ref{fig:strip}).
\begin{figure}[hbt!]
\begin{center}
\fbox{\includegraphics[width=0.6\textwidth]{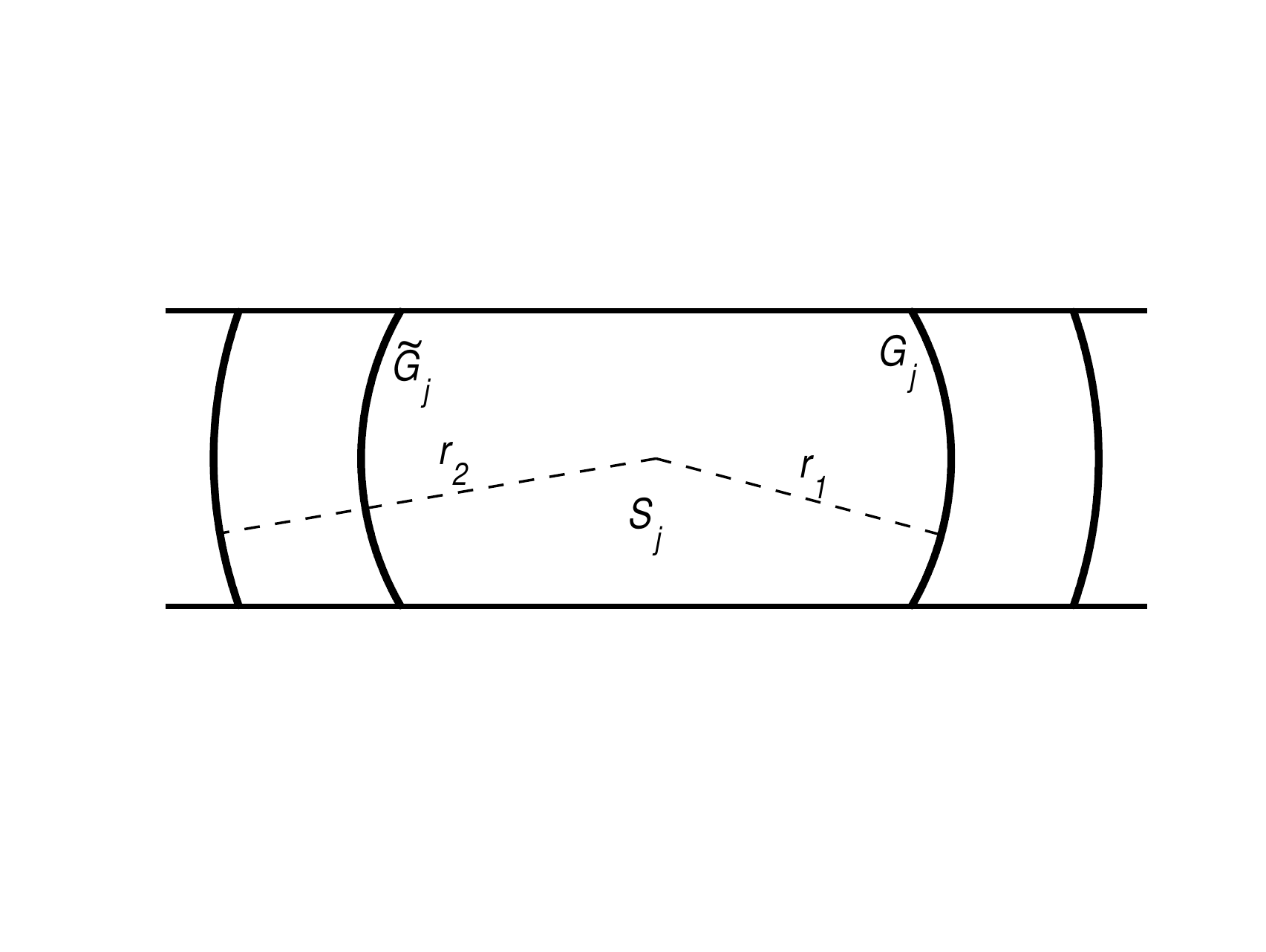}}
\end{center}
\caption{Arcs $G_j$, $\tilde G_j$ and strip $S_j$.}\label{fig:strip}
\end{figure}

Then a little thought shows that the convexity of $\Omega$ implies
\[
(\Omega\cap\{|\xb|\ge r_1\})\subset (\cup_j S_j).
\]
Thus,
\[
\eta(r_2)\le\vol_{1}((\cup_j S_j)\cap\{|\xb|=r_2\}).
\]
However, for each $j$ we have:
\[
\vol_{1}( S_j\cap\{|\xb|=r_2\})<2\vol_1 G_j
\]
(see figure Figure \ref{fig:strip}).
Summing this over $j$, we obtain $\eta(r_1)>\eta(r_2)$.

The concavity of $\alpha$ follows immediately from its definition as an integral of $\eta$.
\end{proof}

\begin{remark}
In a similar manner, one can define the analogues of functions  $\eta$ and $\alpha$ in a higher-dimensional setting. Unfortunately, in general,
the function $\eta$ is no longer decreasing on the interval $[r_-, r_+]$; the simplest counterexample is a strip
$\Omega=\{\xb=(x_1,x_2,x_3)\in\Rbb^3,\, |x_1|<1\}$.
\end{remark}

%%%%%%%%%%%%%%%%%%%%%%%%%%%%%%%%%%%%%%%%%%%%%%%%%%%%%%%%%%%%

\section{Proof of Theorem~\ref{thm:kap_ball}}\label{sec:main_proof}

Set
\[
\tau:=2j_{0,1}\,.
\]
Without loss of generality we assume that
\begin{equation}\label{eq:vol_norm}
\vol_2(\Omega)=\pi\tau^2=4\pi j_{0,1}^2\,,
\end{equation}
and so $\Omega^*=B(\tau)$.
Thus,
\[
\kappa(\Omega^*)=\frac{j_{1,1}}{\tau}=\frac{j_{1,1}}{2j_{0,1}}=\frac{1}{\tilde{C}}\,
\]
with $\tilde{C}$ as in \eqref{eq:C_js}, and in order to prove Theorem~\ref{thm:kap_ball}, we need to prove
\begin{equation}\label{eq:kappa_via_j}
\kappa(\Omega)\le 1\,.
\end{equation}

We prove \eqref{eq:kappa_via_j}, and therefore Theorem~\ref{thm:kap_ball} by a sequence of Lemmas. Some of them are rather
technical, and for convenience the proofs of these Lemmas are collected in the next section.

First, Theorem \ref{thm:diamest} implies that if $D(\Omega)\ge 4\pi$, then the statement is proved.

Correspondingly, if the half-breadth of $\Omega$ in some direction $\eb$,
 $w(\eb)<j_{0,1}^2/2$, then, by Theorem \ref{thm:diamest}, the statement is proved since $2r_{-} D \ge \rm{Vol}_2(\Omega)$ as in Corollary \ref{cor:small_r}.
Thus without loss of generality we can assume that
\begin{equation}\label{eq:r_plus_bound}
r_+=D/2<2\pi\,
\end{equation}
and
\begin{equation}\label{eq:r_minus_bound}
r_- > \frac{j_{0,1}^2}{2}\,.
\end{equation}

The following averaging result is one of the central points of the proof.
\begin{lem}\label{lem:aver}
Suppose that
\begin{equation}\label{eq:intineq1}
\int_\Omega J_0(|\xb|)\,\dr\xb\le 0\,.
\end{equation}
Then \eqref{eq:kappa_via_j} holds.
\end{lem}

\begin{proof}[Proof of Lemma \ref{lem:aver}]
To prove \eqref{eq:kappa_via_j}, it is enough to show that there exists $\eb\in S^1$ such that
\[
\int_\Omega \cos(x_{\eb})\,\dr\xb \le 0\,.
\]

Suppose this inequality is wrong for all $\eb\in S^1$.
Then
\begin{equation}\label{eq:intineq}
\int_{S^1}\int_\Omega \cos(x_{\eb})\,\dr\xb\,\dr\eb > 0\,.
\end{equation}

Changing the order of integration and acting as in in the proof of \eqref{eq:kappa_for_ball}, we get
\[
\int_\Omega J_0(|\xb|)\,\dr\xb>0\,.
\]
Now the Lemma follows by contradiction.
\end{proof}

We now show that the condition of Lemma \ref{lem:aver} in fact follows from  some integral
inequality being satisfied by a class of functions. Namely, consider a class $\mathcal{A}$ of continuous functions
$\alpha:[0,\infty)\to\Rbb$ with  the following properties:

\begin{itemize}
\item[(a)] $ \alpha(r)$ is non-negative and non-decreasing;
\item[(b)] $\alpha(r)=r^2/(4j_{0,1}^2)$ for $0\le r\le r_-$;
\item[(c)] $\alpha(r)=1$ for $r\ge r_+$;
\item[(d)] $\talp(\tilr)$ is concave for $\trm\le \tilr \le\tilr_+$;
\item[(e)] $j_{0,1}^2/2 < r_- \le 2j_{0,1}\le \tilr_+ < 2\pi$.
\end{itemize}

\begin{lem}\label{lem:tildetau} If
\begin{equation}
\label{eq:int_tilalpha}
\sup_{\alpha\in\mathcal{A}}\int_0^{j_{0,3}} \talp(\tilr)J_1(\tilr) \dr\tilr\le 0\,
\end{equation}
holds, then \eqref{eq:kappa_via_j} holds for all planar convex balanced domains $\Omega$ normalized by \eqref{eq:vol_norm}.
\end{lem}

\begin{proof}[Proof of Lemma \ref{lem:tildetau}]
We continue the calculations in the proof of Lemma  \ref{lem:aver}.
Using the geometric notation introduced in the previous Section, we have:
\[
\int_\Omega J_0(|\xb|)\,\dr\xb=\int_0^\infty \eta(r)J_0(r)\,\dr r=
\vol_2(\Omega)\int_0^\infty \alpha(r) J_1(r)\dr r
\]
(the last identity is proved by integration by parts using $J_0'= - J_1$).

By Lemma  \ref{lem:aver}, we need to show that
\[
\int_0^\infty \alpha(r) J_1(r)\dr r\le 0\,.
\]

If for some $k\in\Nbb$, $\alpha(j_{0,k})=1$, then $\alpha(r)=1$ for $r\ge j_{0,k}$, and so, after integration by parts,
\[
\int_{j_{0,k}}^\infty \alpha(r) J_1(r)\dr r  = \int_{j_{0,k}}^\infty J_1(r)\dr r = J_0(j_{0,k}) = 0\,.
\]

We need to choose which $k$ to take. In our case $\alpha(r)=1$ whenever $r\ge D/2$, so by \eqref{eq:r_plus_bound} we need to choose $k$ such that $j_{0,k}>2\pi$ and we can take $k=3$, see the Table \ref{tab:horiz} below.

Thus, we need to show that
\[
I:=\int_0^{j_{0,3}} \alpha(r) J_1( r)\dr r\le 0\,.
\]

The conditions (a)--(d) are just the re-statement of the properties of the function $\alpha$ summarized at the start of  the previous Section with account of normalization \eqref{eq:vol_norm}; condition (e) re-states  \eqref{eq:r_minus_bound}, \eqref{eq:r_plus_bound}, and also the obvious inequalities $\pi r_-^2\le\vol_2(\Omega)\le\pi r_+^2$.
\end{proof}

It is useful here to plot the function $J_1(\tilr)$ and other quantities appearing above.
\begin{figure}[hbt!]
\begin{center}
\fbox{\includegraphics[width=0.8\textwidth]{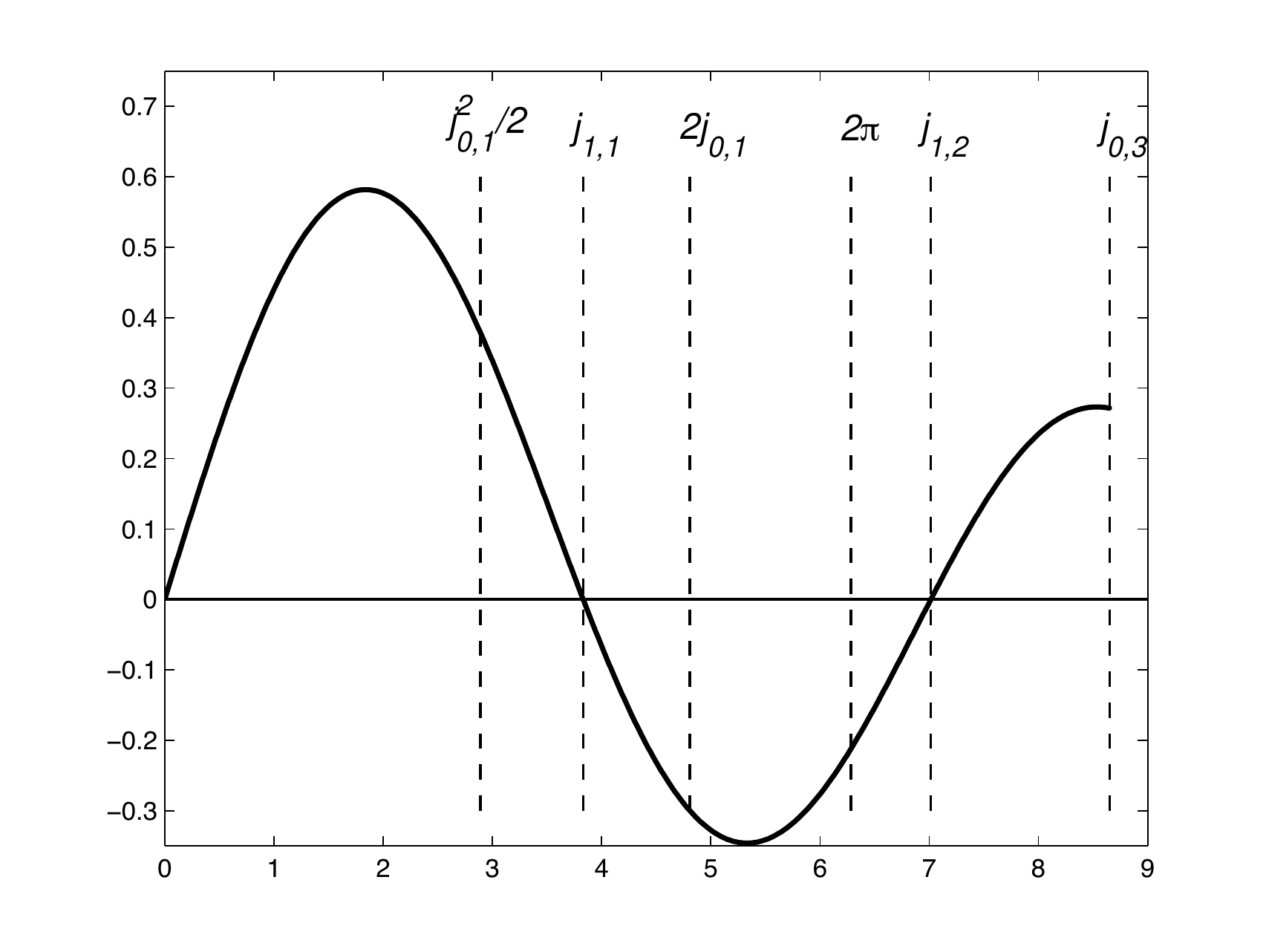}}
\caption{Graph of $J_1(\tilr)$.}\label{fig:J1}
\end{center}
\end{figure}

For future use, we also give two tables of approximate decimal values of various constants
appearing here and below. The first one lists the values appearing along the horizontal axis in various graphs, and the second one lists the values
along the vertical axis. In both Tables the values are sorted out in increasing order.

\begin{table}[hbt!]
\begin{center}
\begin{tabular}{|r|r@{.}l|}
\hline
$2\pi-\sqrt{4\pi^2-\tau^2}$& 2&240206980\\
$\tau^2/8=j_{0,1}^2/2$& 2&891592982\\
$j_{1,1}$& 3&831705970\\
$\tau=2j_{0,1}$& 4&809651116\\
$2\pi$& 6&283185308\\
$j_{1,2}$ & 7&015586670\\
$j_{0,3}$ & 8&653727913 \\
$2\pi+\sqrt{4\pi^2-\tau^2}$& 10&326163640\\
\hline
\end{tabular}
\caption{Decimal values of constants appearing along the horizontal axis.}\label{tab:horiz}
\end{center}
\end{table}

\begin{table}[hbt!]
\begin{center}
\begin{tabular}{|r|r@{.}l|}
\hline
$L$ (see \eqref{eq:L_const})&-0&0852948043\\
$y_{\mathrm{min}}L+M$ (see \eqref{eq:est_final})&-0&0072444612\\
$M$ (see \eqref{eq:M_const})&0&0386824043\\
$c(j_{1,1}^2/\tau^2)=\frac{\tau^2-j_{1,1}^2}{\tau^2(2\pi-j_{1,1})}$ (see \eqref{eq:c_and_d})&0&3403496255\\
$y_{\mathrm{min}}$ (see \eqref{eq:ymin_def})&0&5384485717\\
$j_{1,1}^2/\tau^2$&0&6346834915\\
\hline
\end{tabular}
\end{center}
\caption{Decimal values of constants appearing along the vertical axis.}\label{tab:vert}
\end{table}

The key points of the proof are the estimates of the function
$\talp(\tilr)$ which are collected in the following sequence of Lemmas.

We start by denoting $y_{1,1}:=\talp(j_{1,1})$, and we also intoduce a new constant
\begin{equation}\label{eq:ymin_def}
y_{\mathrm{min}}:=1-\frac{(2\pi-j_{1,1})(64-\tau^2)}{8(16\pi-\tau^2)}\,.
\end{equation}

\begin{lem}\label{lem:y11_est}
For the functions $\talp(\tilr)$ satisfying the conditions (a)--(e) above,
\[
y_{1,1}\ge y_{\mathrm{min}}\,.
\]
\end{lem}

The proof of this Lemma is in the next section.

Now, given the function $\talp(\tilr)$  and using the value of $y_{1,1}=\talp(j_{1,1})$ as a parameter, we construct two new functions. One of them is a linear function
$v(\tilr)=c(y_{1,1})\tilr+d(y_{1,1})$, where the coefficients $c$ and $d$ are chosen to be
\begin{equation}\label{eq:c_and_d}
c=c(y_{1,1})=\frac{1-y_{1,1}}{2\pi-j_{1,1}}\,;\qquad
d=d(y_{1,1}):= 1-2\pi c = \frac{2\pi y_{1,1}-j_{1,1}}{2\pi-j_{1,1}}\,.
\end{equation}
The graph of $v(\tilr)$ is a straight line joining the points
$A_1=(j_{1,1},\talp(j_{1,1}))=(j_{1,1},y_{1,1})$ and $A_+=(2\pi, \talp(2\pi))=(2\pi,1)$.

The other function is a piecewise-continuous one given by
\begin{equation}\label{eq:alpha_approx}
\talp_{\mathrm{approx}}(\tilr):=
\begin{cases}
\tilr^2/\tau^2\quad&\text{for } \tilr\in[0, \tau^2/8]\,;\\
y_{1,1}&\text{for } \tilr\in(\tau^2/8, j_{1,1}]\,;\\
v(\tilr)\quad&
\text{for } \tilr\in[j_{1,1}, 2\pi]\,;\\
1\quad&\text{for } \tilr\in[2\pi,j_{0,3}]\,.
\end{cases}
\end{equation}

\begin{figure}[hbt!]
\begin{center}
\fbox{\includegraphics[width=0.8\textwidth]{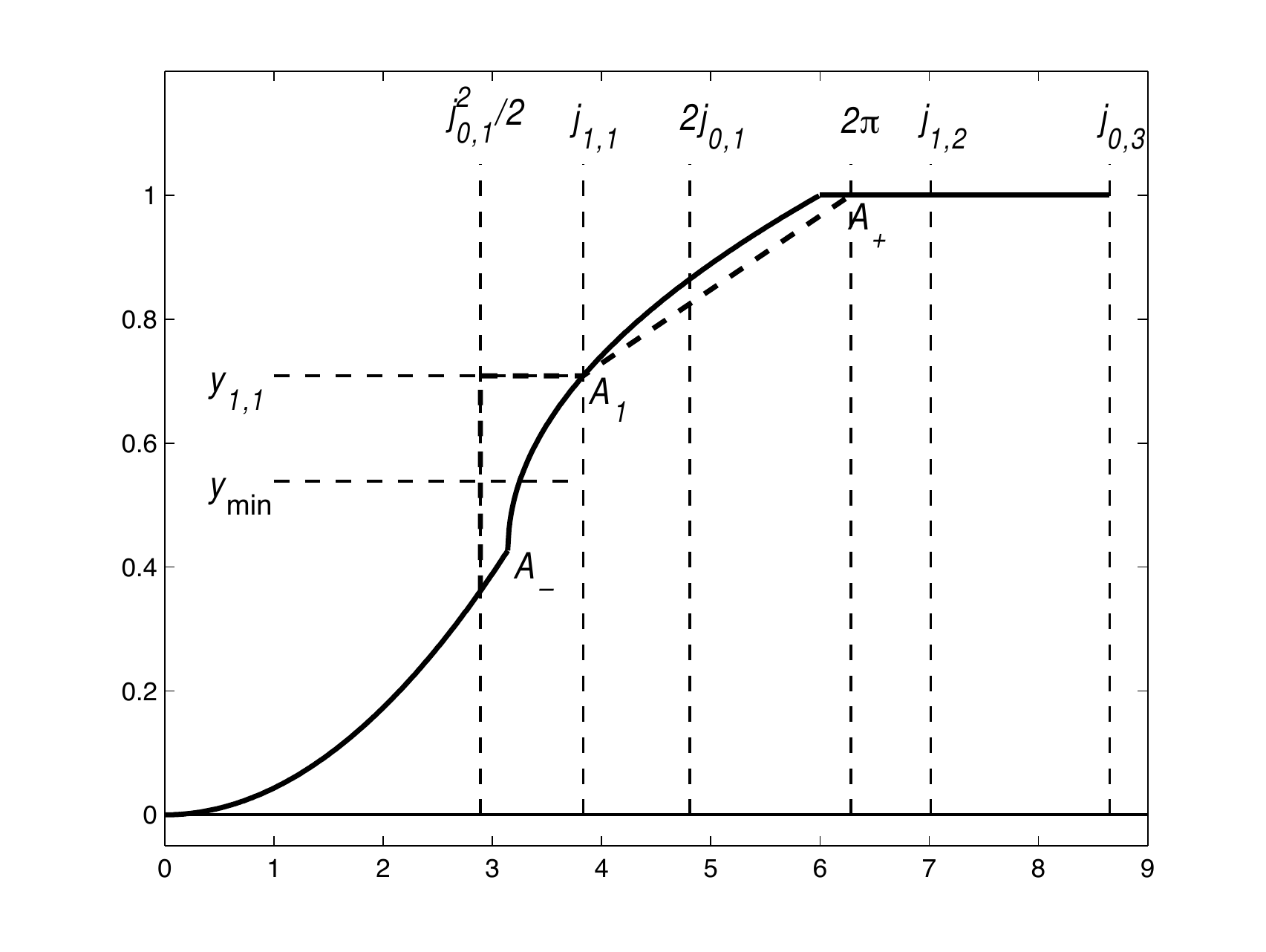}}
\end{center}
\caption{A typical graph of $\talp(\tilr)$ (solid line) and $\talp_{\mathrm{approx}}(\tilr)$ (dashed line). The points have coordinates $A_-=(\trm,\talp(\trm))=(\trm,\trm^2/\tau^2)$, $A_1=(j_{1,1},\talp(j_{1,1}))=(j_{1,1},y_{1,1})$, and $A_+=(2\pi, \talp(2\pi))=(2\pi,1)$.}\label{fig:alpha}
\end{figure}

Obviously, $\talp(\tilr)\equiv \talp_{\mathrm{approx}}(\tilr)\equiv 1$ for $\tilr\ge 2\pi$.

\begin{lem}\label{lem:alp_ge_alp_approx}
Let $\talp(\tilr)$ satisfy conditions (a)--(e). Then
\begin{equation}\label{eq:lessj11}
\talp(\tilr)\le\talp_{\mathrm{approx}}(\tilr)\qquad\text{for }\tilr\in[0, j_{1,1}]\,.
\end{equation}
and
\begin{equation}\label{eq:grj11}
\talp(\tilr)\ge\talp_{\mathrm{approx}}(\tilr)\qquad\text{for }\tilr\in [j_{1,1}, 2\pi]\,.
\end{equation}
\end{lem}

The proof of this Lemma is in the next section.

Lemma \ref{lem:alp_ge_alp_approx} immediately implies, with account of the fact that
$J_1(\tilr)$ changes sign from plus to minus at $\tilr = j_{1,1}$, the following

\begin{cor}\label{cor:int_ineq}
\begin{equation}\label{eq:int_ineq}
\int_0^{j_{0,3}} \talp(\tilr) J_1(\tilr) \dr \tilr
\le \int_0^{j_{0,3}} \talp_{\mathrm{approx}}(\tilr) J_1(\tilr) \dr \tilr\,.
\end{equation}
\end{cor}

The integral in the right-hand side of \eqref{eq:int_ineq} can be explicitly calculated as a function
of the parameter $y_{1,1}$, although the expressions are quite  complicated.
We introduce two constants,
\begin{equation}\label{eq:L_const}
\begin{split}
L&:=J_0\left(\frac{\tau^2}{8}\right)\\
&-\frac{1}{2\pi-j_{1,1}}
\Biggl(\pi^2 J_1(2\pi)\mathbf{H}_0(2\pi)-\pi^2 J_0(2\pi)\mathbf{H}_1(2\pi)+
\frac{\pi j_{1,1}}{2}J_0(j_{1,1})\mathbf{H}_1(j_{1,1})\\
&\quad+j_{1,1}J_0(j_{1,1})+2\pi J_0(2\pi)\Biggr)
\end{split}
\end{equation}
and
\begin{equation}\label{eq:M_const}
\begin{split}
M&:=\frac{1}{8}J_2\left(\frac{\tau^2}{8}\right)\\
&+\frac{1}{2\pi-j_{1,1}}
\Biggl(\pi^2 J_1(2\pi)\mathbf{H}_0(2\pi)-\pi^2 J_0(2\pi)\mathbf{H}_1(2\pi)+
\frac{\pi j_{1,1}}{2}J_0(j_{1,1})\mathbf{H}_1(j_{1,1})\\
&\quad-j_{1,1}J_0(j_{1,1})+2\pi J_0(2\pi)\Biggr)\,;
\end{split}
\end{equation}
in the above formulae $\mathbf{H}$ denote the Struve functions \cite[Chapter 12]{AbrSte}.
The numerical values of $L$ and $M$ can be found in the table above.

\begin{lem}\label{lem:approx_int}
\[
\int_0^{j_{0,3}} \talp_{\mathrm{approx}}(\tilr) J_1(\tilr) \dr \tilr = Ly_{1,1} + M\,.
\]
\end{lem}

With account of Lemmas \ref{lem:tildetau}, \ref{lem:y11_est}, and \ref{lem:approx_int}, and Corollary \ref{cor:int_ineq} we immediately have
\begin{equation}
\label{eq:est_final}
\int_0^{j_{0,3}} \talp(\tilr) J_1(\tilr) \dr \tilr \le Ly_{\mathrm{min}} + M\approx -0.00724446126 <0\,,
\end{equation}
which finishes the proof of the Theorem.

%%%%%%%%%%%%%%%%%%%%%%%%%%%%%%%%%%%%%%%%%%%%%%%%%%%%%%%%%%%%

\section{Proofs of Lemmas}\label{sec:lemmas}

\begin{proof}[Proof of Lemma \ref{lem:y11_est}]
There are two possibilities. If $\trm\ge j_{1,1}$, then
$\talp(j_{1,1})=\frac{j_{1,1}^2}{\tau^2}$ by condition (b), and the claim of the Lemma is true.
We thus need to consider a case when  $\trm\le j_{1,1}$.

Let us introduce a linear function $y(\tilr):=a(\trm)\tilr+b(\trm)$, where the constants $a$ and $b$ depend
upon $\trm$ as a parameter and are chosen to be
\begin{equation}\label{eq:a_and_b}
a=a(\trm)=\frac{\tau^2-\trm^2}{\tau^2(2\pi-\trm)}\,;\qquad
b =b(\trm):= 1-2\pi a = \frac{\trm(2\pi\trm -\tau^2)}{\tau^2(2\pi-\trm)}\,.
\end{equation}
The graph of $y(\tilr)$ is a straight line joining the points
$A_-=(\trm,\talp(\trm))=(\trm,\trm^2/\tau^2)$ and $A_+=(2\pi, \talp(2\pi))=(2\pi,1)$.

The function $\talp(\tilr)$ is concave on the interval $[\trm, 2\pi]$ by conditions (c) and (d),
and its graph passes through the points $A_-$ and $A_+$. Thus, this graph lies above the straight line joining
$A_-$ and $A_+$, and therefore
\begin{equation}\label{eq:conc_gr_str}
\talp(\tilr)\ge y(\tilr)\qquad\text{for }\tilr\in[\trm, 2\pi]\,.
\end{equation}

As $j_{1,1}\in[\trm, 2\pi]$, \eqref{eq:conc_gr_str} implies
\[
\talp(j_{1,1})\ge a(\trm)j_{1,1}+b(\trm)=1-(2\pi-j_{1,1})a(\trm)\,,
\]
and as in our case $\trm$ can take values only in the interval $[\tau^2/8, j_{1,1}]$, we have
\[
\talp(j_{1,1})\ge 1-(2\pi-j_{1,1})\max\limits_{\trm\in[\tau^2/8, j_{1,1}]} a(\trm)
\]

We have
\[
\frac{\dr a(\trm)}{\dr \trm}=\frac{\trm^2-4\pi\trm+\tau^2}{\tau^2(2\pi-\trm)^2}\,.
\]
The roots of the numerator in the right-hand side are $2\pi\pm\sqrt{4\pi^2-\tau^2}$, and as seen from the table above the derivative is negative for $\trm\in[\tau^2/8, j_{1,1}]$.
Thus
\begin{equation}\label{eq:y11est}
y_{1,1}\ge 1-(2\pi-j_{1,1})a(\tau^2/8)\,.
\end{equation}
It is an easy manipulation to check that the right-hand side of \eqref{eq:y11est} equals $y_{\mathrm{min}}$.
\end{proof}

\begin{proof}[Proof of Lemma \ref{lem:alp_ge_alp_approx}]
As $\talp_{\mathrm{approx}}(\tilr)=\talp(\tilr)$ for $\tilr$ in the interval $[0, \tau^2/8]$, and $\talp_{\mathrm{approx}}(\tilr)=y_{1,1}=\talp(j_{1,1})$ for $\tilr$ in the interval
$[\tau^2/8, j_{1,1}] $, inequality \eqref{eq:lessj11} follows immediately from the monotonicity
condition (a).

In order to prove \eqref{eq:grj11}, we again need to consider two cases. First,
if $\trm< j_{1,1}$, then $\talp(\tilr)$ is concave for $\tilr\in [j_{1,1}, 2\pi]$, and its graph
between the points $A_1$ and $A_+$ lies above the straight
line joining this points. Thus, it remains to consider the case $\trm\ge j_{1,1}$ (and so
$y_{1,1}=j_{1,1}^2/\tau^2$).

We now  show  that in this case
\begin{equation}\label{eq:r2_ge_alp_approx}
\frac{\tilr^2}{\tau^2}\ge v(\tilr)=c(j_{1,1}^2/\tau^2)\tilr+d(j_{1,1}^2/\tau^2)
\qquad\text{for }\tilr\ge j_{1,1}\,.
\end{equation}
Indeed, consider
\[
u(\tilr):=\frac{\tilr^2}{\tau^2}-c(j_{1,1}^2/\tau^2)\tilr-d(j_{1,1}^2/\tau^2)\,.
\]
We have
\[
u(j_{1,1})=0\,,
\]
and also for $\tilr\ge j_{1,1}$,
\[
\frac{\dr u}{\dr\tilr}(\tilr)=2\frac{\tilr}{\tau^2}-c(j_{1,1}^2/\tau^2)\ge
2\frac{j_{1,1}}{\tau^2}-\frac{\tau^2-j_{1,1}}{\tau^2(2\pi-j_{1,1})}>0
\]
(see table above for numerical values), which proves \eqref{eq:r2_ge_alp_approx}.

Thus, $\talp(\trm)\ge\talp_{\mathrm{approx}}(\trm)$.  As, by concavity, the graph of
$\talp(\tilr)$ between the points $A_-$ and $A_+$ lies above the straight line joining these points,
and the graph of $\talp_{\mathrm{approx}}(\tilr)$ is a straight line joining the point
$(\trm, \talp_{\mathrm{approx}}(\trm))$ (which is located below $A_-$) with $A_+$,
inequality \eqref{eq:grj11} follows.
\end{proof}

\begin{proof}[Proof of Lemma \ref{lem:approx_int}]
The result follows from straightforward integration of \eqref{eq:alpha_approx} using the
standard relations
\begin{align*}
\int J_1(x)\dr x &= - J_0(x)\qquad\qquad\qquad\text{\cite[formula 11.1.6]{AbrSte}}\,;\\
\int xJ_1(x)\dr x &= -\int xJ_0'(x)\dr x=-xJ_0(x)+\int J_0(x)\dr x\\
 &= \frac{\pi x}{2}\,\left(J_1(x)\mathbf{H}_0(x)-J_0(x)\mathbf{H}_1(x)\right)
\quad \text{\cite[formula 11.1.7]{AbrSte}}\,;\\
\int x^2J_1(x)\dr x &= x^2J_2(x)\qquad\qquad\qquad \text{\cite[formula 11.3.20]{AbrSte}}\,.
\end{align*}
\end{proof}

%%%%%%%%%%%%%%%%%%%%%%%%%%%%%%%%%%%%%%%%%%%%%%%%%%%%%%%%%%%%

\section{Perturbation-type results}\label{sec:perturbations}

In this Section, we prove Theorem \ref{thm:asympt_est}. In order to do this, we need to compute the one-sided derivatives of $\sqrt{\lambda_2(\Omega_{\eps F})}$ and
$\kappa(\Omega_{\eps F})$ with respect to the parameter $\eps$ describing the deformations of the disk.
The first derivative is easily computable from the following classical result  (see e.g., \cite{Hen,Rel}) which we state without proof:

\begin{thm} [Derivative of a multiple Dirichlet eigenvalue]\label{thm:t1}
Let $\Omega_0\subset\mathbb{R}^d$ be a bounded domain with $C^2$ boundary. Assume that
$\lambda_k(\Omega_0)=\dots=\lambda_{k+p-1}(\Omega_0)$ is a multiple Dirichlet eigenvalue of order $p \ge 2$. Let us
denote by $u_{k_1}, u_{k_2}, \dots, u_{k_p}$ an orthonormal family
of eigenfunctions associated to $\lambda_k$. Let $\mathbf{S}(t):\mathbb{R}^d\to\mathbb{R}^d$, $t\in[0, t_0)$, be a continuously differentiable with respect to  $t$
family of mappings such that $\mathbf{S}(0)$ is an identity, and let $\Omega_t=\mathbf{S}(t)(\Omega_0)$. Then,
the function $t \to \lambda_k(\Omega_t)$ has a (directional) derivative at $t=+0$
which is one of the eigenvalues of the $p \times p$ matrix $M=[m_{i,j}] $
defined by
\begin{equation}\label{eq:der1}
m_{i,j} = - \int_{\partial \Omega} \left(
\frac{\partial u_{k_i}}{\partial n} \frac{\partial u_{k_j}}{\partial
n} \right) \mathbf{S}'(0)(\sigma)\cdot \mathbf{n}(\sigma) \, \dr \sigma, \qquad i,j=1,\dots p\,,
\end{equation}
where $\mathbf{n}(\sigma)$ is an exterior normal to $\partial\Omega$ at the point $\sigma\in\partial\Omega$.
\end{thm}

Theorem \ref{thm:t1} implies,

\begin{lem}\label{lem:pert_ball} Let $\Omega_{\eps F}$ be as in Theorem \ref{thm:asympt_est}. Then
\begin{equation} \label{eq:der0}
\left.\frac{\dr \sqrt{\lambda_2(\Omega_{\eps F})}}{\dr\eps}\right|_{\eps=+0} =
-\frac{j_{1,1} }{2\pi} \left \vert \int_0^{2\pi} F(\theta) e^{2i \theta}\,\dr\theta\right \vert\,.
\end{equation}
\end{lem}

\begin{proof}[Proof of Lemma \ref{lem:pert_ball}]
In our case, $\Omega_0=\Omega_{0F}$ is a disk of radius $1$,
$\lambda_2(\Omega_0)$ is doubly degenerate, so the matrix $M$ is of
dimension $2$, and we can choose the  orthonormal eigenfunctions 
$u_2(r, \theta)=N\, J_1(j_{1,1}  r) \cos \theta$, $u_3(r, \theta)= N\, J_1(j_{1,1} r)
\sin \theta$. The constant $N$ is introduced in order for the 
eigenfunctions $u_2$ and $u_3$ to have $L_2$--norm one in the unit disk. Using standard 
properties of Bessel functions (in particular \cite[formulas 11.45 and 9.1.30]{AbrSte}) one gets, 
$$
N=\sqrt{\frac{2}{\pi}} \frac{1}{|J_0(j_{1,1})|}.
$$
Using the lowering property of Bessel functions, 
$$
\frac{1}{z} \frac{\dr}{\dr z} \left(z \, J_1(z) \right) =  J_0,
$$
(\cite[formula 9.1.30]{AbrSte})
the value of $N$ just obtained, and the fact that $J_1(j_{1,1})=0$ and $J_0(j_{1,1})<0$, 
in the expresion for $u_2$ and $u_3$ we obtain,
$$
\frac{\partial u_2}{\partial n} = -j_{1,1}\sqrt{\frac{2}{\pi}} \cos \theta,
$$
and
$$
\frac{\partial u_3}{\partial n} = -j_{1,1}  \sqrt{\frac{2}{\pi}} \sin \theta,
$$
at the boundary of the disk (i.e., at $r=1$).

Taking these remarks into account, the elements of the matrix $M$ in
our case are given by
\[
m_{1,1} = \frac{2 j_{1,1} ^2}{\pi} \int_0^{\pi} F(\theta) \cos^2 \theta\, \dr \theta,
\]
\[
m_{1,2}=m_{2,1} = \frac{2 j_{1,1} ^2}{\pi} \int_0^{\pi} F(\theta) \cos\theta \sin \theta\, \dr \theta,
\]
and
\[
m_{2,2} = \frac{2 j_{1,1} ^2}{\pi} \int_0^{\pi} F(\theta) \sin^2 \theta\, \dr \theta\,.
\]
For area preserving deformations of the ball, i.e. for functions $F$ satisfying \eqref{eq:intF0}, we can write the matrix
$M$ in the simple form,
\[
M =
\begin{pmatrix} \re a & \im a \\
\im a & -\re a
\end{pmatrix}
\]
where
\[
a=\frac{j_{1,1}^2}{\pi}\int_0^{2\pi} \er^{2\ir\theta} F(\theta)\, \dr \theta\,.
\]
It is
simple to compute the two eigenvalues of $M$ in this case, and they
are given by
\[
\pm  \vert a \vert.
\]
Hence, using Theorem \ref{thm:t1}  we obtain  the (directional)
derivative of  the second eigenvalue of the perturbed
domain by using a smaller of these two eigenvalues:
\begin{equation}
\frac{d \lambda_2(\Omega_{\eps F})}{dt}\Bigm|_{t=0} = -\frac{j_{1,1} ^2}{\pi}
\left \vert \int_0^{2\pi} F(\theta) \er^{2\ir \theta}\,\dr\theta \right \vert\,.
\label{eq:der3}
\end{equation}
From (\ref{eq:der3}), taking into account that $\lambda_2=j_{1,1}^2 $, we
finally get (\ref{eq:der0}).
\end{proof}

Now, we will compute the derivative of $\kappa(\Omega_{\eps F})$ at
$\eps=0$, for area preserving deformations of the disk.
For brevity, we shall use the notation
$f_\eps(\xib):=\widehat{\chi_{\Omega_{\eps F}}}(\xib)$ for the Fourier transform of the characteristic function of
$\Omega_{\eps F}$ and
$\NC_\eps=\NC(\Omega_{\eps F})=\{\xib\in\Rbb^d:f_\eps(\xib)=0\}$ for its null variety.

We know that $\NC_0$ contains  a circle of radius $j_{1,1}$, and we
seek to characterize the elements of $\NC_\eps$. Pick an element
$\xib_0 \in {\cal N}_0$. For definiteness, we choose coordinates
in such a way that
\begin{equation}\label{eq:xi0dir}
\xib_0=(1,0) j_{1,1} ,
\end{equation}
and we write an element of ${\cal N}_{\eps}$ as
$$
\xib =\xib_0 + \eps\xib_1.
$$
It is precisely ${\xib}_1$ which we would like to determine by
requiring  $f_\eps(\xib)=0$ to hold up to first order in $\eps$. Using
polar coordinates, we write
$$
{\xib}_1 = (\cos \omega, \sin \omega) \rho_1.
$$
With the above notation, we have
\begin{equation}\label{eq:per6}
f_\eps(\xib) = \int_0^{2 \pi} \int_0^{1+\eps F(\theta)} \er^{\ir j_{1,1}  r
\cos \theta}\er^{\ir \eps \rho_1 r \cos (\theta- \omega)} \, r \, \dr r \, \dr\theta\,.
\end{equation}
In the sequel, we use the fact that
$f_{0}(j_{1,1} )=0$, i.e.,
\begin{equation}\label{eq:per7}
\int_0^{2\pi} \int_0^1 e^{\ir j_{1,1}  r \cos \theta} \, r\, \dr r \,
\dr\theta=0\,,
\end{equation}
and split the integral in the variable $r$ in \eqref{eq:per6} as an
integral from $r=0$ to $r=1$ plus an integral from $r=1$ to
$r=1+\eps F(\theta)$. After some algebraic computations we obtain
\begin{equation} \label{eq:per8}
f_\eps (\xib) = \eps N(\rho_1, \omega)+O(\eps^2),
\end{equation}
where
\begin{equation}\label{eq:per9}
\begin{split}
N(\rho_1,\omega) &= \ir \rho_1 \int_0^{2\pi} \left( \int_0^1 \er^{\ir j_{1,1}  r\cos
\theta} r^2 \cos (\theta-\omega) \, \dr r \right)\\
&+ \int_0^{2\pi}
F(\theta) \er^{\ir j_{1,1}  \cos \theta} \, \dr \theta\,.
\end{split}
\end{equation}
Using the fact that the perturbed domain is {\it balanced}, i.e.,
that $F(\theta)=F(\theta+\pi)$, we get
$$
\int_0^{2\pi} F(\theta) e^{\ir j_{1,1}  \cos \theta} \, \dr \theta=
\int_0^{2\pi} F(\theta) \cos(j_{1,1}  \cos \theta) \, \dr \theta.
$$
Since we also have $\int_0^{2\pi} \er^{\ir j_{1,1}  r \cos \theta} \sin
\theta \, \dr \theta=0$, we arrive at
\begin{equation}\label{eq:per10}
\begin{split}
N(\rho_1,\omega) &= \ir \rho_1 \int_0^{2\pi} \left( \int_0^1 \er^{\ir j_{1,1}  r\cos
\theta} \cos \theta \cos \omega \,r^2 \dr r \right)\\
&+ \int_0^{2\pi}
F(\theta) \cos(j_{1,1}  \cos \theta) \, \dr \theta\,.
\end{split}
\end{equation}
Integrating the first integral in \eqref{eq:per10}  by parts in $r$
gives
\begin{equation}\label{eq:per11}
\begin{split}
N(\rho_1,\omega) &= \frac{\rho_1}{j_{1,1} } \cos \omega\left( \int_0^{2\pi} \er^{\ir j_{1,1} \cos \theta} \, \dr \theta
- \int_0^1 \int_0^{2\pi} 2 r \er^{\ir j_{1,1}  r \cos \theta} \, \dr r \, \dr \theta \right)\\
&+ \int_0^{2\pi} F(\theta) \cos(j_{1,1} \cos \theta) \, \dr \theta.
\end{split}
\end{equation}
Now, we use the integral representation
$$
J_0(x) = \frac{1}{2\pi} \int_0^{2\pi} \er^{\ir x \cos \theta} \, \dr \theta
$$
to simplify the first two terms in (\ref{eq:per11}). We finally get
\begin{equation}\label{eq:per12}
N(\rho_1,\omega) = 2 \pi  \frac{\rho_1}{j_{1,1} } \cos \omega\, J_0(j_{1,1} ) +
\int_0^{2\pi} F(\theta) \cos(j_{1,1}  \cos \theta) \, \dr \theta.
\end{equation}
Here we have used the fact that $\int_0^1 r J_0(j_{1,1}  r) \, \dr r=
J_1(j_{1,1} )/j_{1,1} ^2=0$, since $j_{1,1}$ is a zero of $J_1$.
The vector $\xib_1=\rho_1(\cos \omega, \sin \omega)$ is determined by the
condition
$$
N(\rho_1,\omega)=0\,.
$$
Therefore, (\ref{eq:per12}) implies
\begin{equation}\label{eq:per13}
\rho_1 \cos \omega = - \frac{j_{1,1} }{2 \pi J_0(j_{1,1} )} \int_0^{2\pi} F(\theta)
\cos(j_{1,1}  \cos \theta) \, \dr \theta.
\end{equation}

In the case when $\xib_0$ is not given by \eqref{eq:xi0dir}, but by
\[
\xi_0=(\cos\alpha, \sin\alpha)j_{1,1}\,,
\]
we have
\begin{equation}\label{eq:per13a}
\rho_1 \cos \omega = - \frac{j_{1,1} }{2 \pi J_0(j_{1,1} )} \int_0^{2\pi} F(\theta+\alpha)
\cos(j_{1,1}  \cos \theta) \, \dr \theta.
\end{equation}

In order to compute $\kappa(\Omega_{\eps F})$ to first order in $\eps$ all
we have to compute is $\vert\xib_0 + \eps\xib_1\vert$ to first
order in $\eps$, which in turn is given by
$$
 j_{1,1}  + \eps\rho_1 \cos \omega + O(\eps^2)\,.
$$
Using \eqref{eq:per13a}, we obtain
\begin{equation}
\kappa(\Omega_{\eps F}) = \min_{\alpha}\left[j_{1,1}  \left( 1 - \frac{\eps}{2\pi
J_0(j_{1,1} )} \int_0^{2\pi} F(\theta+\alpha) \cos(j_{1,1}  \cos \theta) \, \dr\theta
\right) + O(\eps^2)\right]. \label{eq:per14}
\end{equation}
From this result, we immediately have,
\begin{lem}\label{lem:kappa_deriv} Let $F$ be a $C^2$ function on a unit circle satisfying periodicity condition
\eqref{eq:F_bal}. Then
\begin{equation} \label{eq:per15}
\left.\frac{\dr \kappa(\Omega_{\eps F})}{\dr\eps}\right|_{\eps=0} = \min_{\alpha} \left[-
\frac{j_{1,1} }{2\pi J_0(j_{1,1} )} \int_0^{2\pi} F(\theta+\alpha) \cos(j_{1,1}  \cos
\theta) \, \dr\theta \right].
\end{equation}
\end{lem}

\begin{remark}
Note that Lemma \ref{lem:kappa_deriv} does not assume the area preservation condition \eqref{eq:intF0}.
\end{remark}

In order to finish the proof of Theorem \ref{thm:asympt_est} for perturbations
around the circle, we need to prove that the right-hand side of
(\ref{eq:per15}) is always less or equal than the right-hand side of
(\ref{eq:der0}), i.e.,
\begin{equation}\label{eq:p1}
\min_{\alpha} \int_0^{\pi} F(\theta+\alpha) \cos (j_{1,1}  \cos \theta)
\, d\theta \le - A \left\vert \int_0^{\pi} F(\theta) e^{2i \theta}
\, \dr\theta\right\vert,
\end{equation}
where $A=-J_0(j_{1,1} ) \approx 0.408 \dots$, assuming additionally that $F$ satisfies the area preservation
condition  \eqref{eq:intF0}.

For future reference we denote the left-hand side and the right-hand side of \eqref{eq:p1} by $L_F$ and $R_F$, respectively.

%\begin{remark} Note that if $F(\theta)=c=\const$ (which is, in fact, not allowed for $c\ne0$ by \eqref{eq:intF0}),
%then (\ref{eq:p1}) holds with equality. In fact, in this case,
%the right-hand side of (\ref{eq:p1}) is given by $R_F=-A\pi c$ whereas the
%left-hand side is given by $L_F=c \int_0^{\pi} \cos (j_{1,1}  \cos \theta) \, \dr
%\theta = c \pi J_0(j_{1,1} )$. Since $A=-J_0(j_{1,1} )$, we have $L_F=R_F$.
%\end{remark}

\begin{remark}
Note also that if the average of $F$ is $0$ and, additionally, $F$
has zero two--modes, i.e., $\int_0^{\pi} F(\theta) \er^{2\ir\theta} \,
\dr \theta =0$, then (\ref{eq:p1}) is valid. In fact, in this case, the
right-hand side $R_F$ vanishes whereas the left-hand side is given
by
$$
L_F=\min_{\alpha} \int_0^{\pi} F(\theta+\alpha) \cos (j_{1,1}  \cos \theta)
\, \dr \theta,
$$
so we have
\begin{equation}\label{eq:p2}
L_F\le\int_0^{\pi} F(\theta+\alpha) \cos (j_{1,1}  \cos \theta) \, \dr \theta,
\end{equation}
for every $\alpha$, and averaging over $\alpha$ we get,
\[
\begin{split}
L_F & \le \frac{1}{\pi} \int_0^{\pi} \left(\int_0^{\pi}
F(\theta+\alpha) \cos (j_{1,1}  \cos \theta) \, \dr \theta \right) \, \dr
\alpha\,,\\
& = \frac{1}{\pi} \int_0^{\pi} \dr \theta \cos(j_{1,1}  \cos \theta)
\int_0^{\pi} F(\theta+ \alpha) \, \dr \alpha =0\,,
\end{split}
\]
and we are done.
\end{remark}

Before we conclude, we need to analize the case of equality in \eqref{eq:former2.10}, i.e., we need to show that equality is only attained in  \eqref{eq:former2.10} if the domain is a ball, or, in other words, if $F(\theta) \equiv 0$. In order to have equality in  \eqref{eq:former2.10}, we need equality in  \eqref{eq:p2}, which in turn implies, 
\begin{equation}
\int_0^{\pi} F(\theta+ \alpha) \cos(j_{1,1} \cos \theta) \, \dr \theta=0,
\label{eq:rem1}
\end{equation}
for all $\alpha$. Since $F(\theta)$ has zero average, and moreover $F(\theta+\pi)=F(\theta)$ (which is required so that the perturbed domain is balanced), the Fourier serries of $F(\theta)$ can be written as
\begin{equation}
F(\theta) = \sum_{k=-\infty}^{\infty} c_k \er^{\ir k\theta} = \sum_{m \neq 0} c_{2m} e^{2\ir m\theta},
\label{eq:rem2}
\end{equation}
with $c_0=0$ (because $F$ has zero average) and $c_{2k+1}=0$, for all $k$ (because the domain is balanced). 
Replacing (\ref{eq:rem2}) in (\ref{eq:rem1}) we get,
\begin{equation}
\sum_{m \neq 0} c_{2m} \int_0^{\pi} \er^{2\ir m(\theta+\alpha)} \cos(j_{1,1} \cos \theta) \, \dr\theta=0,
\label{eq:rem3}
\end{equation}
Using the integral representation for $J_n(z)$, i.e., 
$$
J_n(z)=\frac{1}{\pi} \int_0^{\pi} \cos(z \sin \theta - n \theta) \, d\theta, 
$$
after some computation we can write (\ref{eq:rem3}) as 
\begin{equation}
\sum_{m \neq 0} c_{2m} (-1)^m J_{2m} (j_{1,1}) \er^{2\ir m\alpha}=0,
\label{eq:rem4}
\end{equation}
all $0 \le \alpha \le 2 \pi$. Since the $\exp(2\ir m\alpha)$ form an orthogonal set of functions, 
we finally get,
$$
c_{2m} (-1)^m J_{2m} (j_{1,1})=0,
$$
all $m$. Since $j_{m,1} > j_{1,1}$ for all $m>1$, $J_{2m}(j_{1,1}) \neq 0$, thus, $c_{2m}=0$, all $m$, hence, 
$F(\theta)\equiv 0$ as it was to be shown. 
A similar argument can be used to show that equality is attained only for the ball in the general case.

Before we go into the proof of (\ref{eq:p1}) for a general $F$
satisfying both the periodicity condition (giving a balanced domain) and the zero average condition (area preserving domain perturbation),
we need the following result.

\begin{lemma}\label{lem:Frot}
Assuming $F$ averages up to zero, it is always possible to rotate
$F$ in such a way that the following two
conditions are fulfilled simultaneously:
\begin{equation}\label{eq:Frot0}
\int_0^{\pi} F(\theta+\phi) \sin (2 \theta)\, \dr\theta= 0\,,
\end{equation}
and
\begin{equation}\label{eq:Frot2}
\int_0^{\pi} F(\theta+\phi) \, \cos (2 \theta) \, \dr\theta \ge 0\,.
\end{equation}
Here $F(\theta+ \phi)$ is $F$ rotated by an angle
$\phi$.
\end{lemma}

\begin{proof}[Proof of Lemma \ref{lem:Frot}]
Consider the function
\[
T(\phi):=\int_0^{\pi} F(\theta+\phi) \, \sin (2 \theta) \, \dr \theta\,.
\]
Since $F$ averages to zero, $\int_0^{\pi} T(\phi) \, \dr \phi=0$, so
there exists a point $\phi_1 \in [0,\pi]$, such that
$T(\phi_1)=0$, and \eqref{eq:Frot0} holds.

Now, consider
\[
Q(\phi):= \int_0^{\pi} F(\theta+\phi) \cos (2 \theta) \, \dr\theta\,.
\]
Clearly, $T(\phi_1) = T(\phi_1 + \pi/2)=0$. On the other
hand, $Q(\phi_1)=-Q(\phi_1+\pi/2)$. So, either $Q(\phi_1) \ge 0$, or
$Q(\phi_1+\pi/2)\ge 0$, and we have obtained \eqref{eq:Frot2} by choosing $\phi=\phi_1$ or
$\phi=\phi_1+\pi/2$.
\end{proof}

After proving this Lemma we are ready to prove \eqref{eq:p1}.
Consider $F$ with zero average and such that $\int_0^{\pi}
F(\theta) \cos (2 \theta)\, \dr \theta \ge0$ and $\int_0^{\pi} F(\theta) \sin
(2\theta) \, \dr \theta =0$. In this case, the right-hand side of
(\ref{eq:p1}) is given by
\begin{equation}
R_F=-A \int_0^{\pi} F(\theta) \cos (2 \theta)\, \dr\theta.
\end{equation}

On the other hand, the left-hand side $L_F$ satisfies \eqref{eq:p2} for each $\alpha$.
Now, multiply (\ref{eq:p2}) by
\[
\cos^2 \alpha/\int_0^{\pi} \cos^2
\alpha \, \dr \alpha \equiv (2/\pi) \cos^2 \alpha
\]
and integrate in $\alpha$ from $0$ to $\pi$ (notice that $(2/\pi) \cos^2 \alpha \ge
0$). We thus have,
\begin{equation} \label{eq:p3}
L_F \le \frac{2}{\pi} \int_0^{\pi} \left( \int_0^{\pi} F(\theta+ \alpha)
\cos( j_{1,1}  \cos \theta) \, \dr\theta \right) \cos^2 \alpha \, \dr
\alpha.
\end{equation}
Now, split $\cos^2 \alpha = (\er^{2\ir \alpha} + \er^{-2 \ir \alpha} +2)/4$
in (\ref{eq:p3}). If we do the integral over $\alpha$ first, using
the fact that the average of $F$ vanishes, we get
\[
\begin{split}
\int_0^{\pi} F(\theta+\alpha) \, \cos^2 \alpha \, \dr \alpha  &=
\frac{1}{4} \int_0^{\pi} \er^{2\ir\alpha} F(\theta+\alpha) \, \dr \alpha
+\frac{1}{4} \int_0^{\pi} \er^{-2\ir\alpha} F(\theta+\alpha) \, \dr\alpha \\
&= \frac{1}{4} \er^{-2\ir \theta} \int_0^{\pi} \er^{2\ir\beta} F(\beta) \, \dr
\beta +\frac{1}{4} \er^{2 \ir \theta}\int_0^{\pi} \er^{-2\ir\beta} F(\beta)
\, \dr \beta.
\end{split}
\]
By Lemma \ref{lem:Frot}, and the choice of orientation of $F$, we have
\[
\int_0^{\pi} \er^{2 \ir \beta} F(\beta) \, \dr \beta = \int_0^{\pi} \er^{-2
\ir \beta} F(\beta) \, \dr \beta =\int_0^{\pi} \cos(2 \beta) F(\beta) \,
\dr \beta =: P\,,
\]
and so
\[
\int_0^{\pi} F(\theta+\alpha) \cos^2 \alpha \, \dr \alpha =
\frac{1}{2} \cos (2 \theta) P\,.
\]
Then,
\begin{equation} \label{eq:p7}
L_F \le \frac{2}{\pi} \frac{1}{2} \int_0^{\pi} \cos (2 \theta) \cos
(j_{1,1}  \cos \theta) \, \dr\theta \, P = P J_0(j_{1,1} )\,.
\end{equation}
Here we have used the fact that $\int_0^{\pi} \cos (2 \theta) \cos (j_{1,1}  \cos
\theta) \, \dr\theta = \pi J_0(j_{1,1} )$ (this follows by taking real part in \cite[formula 9.1.21]{AbrSte}, with $n=2$, and the fact 
that $J_2(j_{1,1})=-J_0(j_{1,1})$ \cite[9.1.27]{AbrSte}) . Hence,
\[
L \le J_0(j_{1,1} ) P = - A \, P=R_F.
\]

This proves \eqref{eq:p1} and  therefore Theorem \ref{thm:asympt_est}.\qed

%%%%%%%%%%%%%%%%%%%%%%%%%%%%%%%%%%%%%%%%%%%%%%%%%%%%%%%%%%%%

\section{Non-convex domains: counterexamples}\label{sec:counterex}

We start by proving Theorem~\ref{thm:counterexample}.

First, we introduce some notation. For a domain $\Omega$ we put
\[
\zeta(r)=\zeta_{\Omega}(r):=\frac{\eta(r)}{2\pi r}.
\]
Then the function $\zeta$ satisfies the following properties:
$\zeta(r)\le 1$; if $\Omega$ is star-shaped, $\zeta$ is non-increasing,
$\supp\zeta=\supp\eta =[0,D(\Omega)/2]$, and $\vol_2(\Omega)=2\pi\int_0^{D/2}r\zeta(r)\dr r$.
The strategy of the proof is the following: first, we construct a function $\tilde\zeta$ which is non-increasing,
$\tilde\zeta(r)=1$ for $0\le r\le 1-\tilde\delta$, $\supp\tilde\zeta \subset [0,1+\tilde\delta]$,
$\int_0^{1+\tilde\delta}r\tilde\zeta(r)\dr r=\pi$ and, finally,
\begin{equation}\label{l51}
\int_0^{\infty}r\tilde\zeta(r)J_0(\gamma r)\dr r>0
\end{equation}
for all $\gamma\le j_{1,1}$. Then we construct a domain $\Omega$ such that $\tilde\zeta=\zeta_{\Omega}$ and
$
\int_{\Omega}\cos(\gamma x_{\eb})d\xb
$
is close to the left-hand side of \eqref{l51} for all $\eb$, $|\eb|=1$ and all $\gamma\le j_{1,1}$.
This $\Omega$ will be a required domain.

Let $\zeta_0(r)=\begin{cases} 1,&\ 0<r\le 1\\0,&\ r>1\end{cases}$ be the $\zeta$-function for the ball
of radius one. Suppose that $\tilde\delta$ is fixed. Let $\delta$ be a small positive parameter, and put
\[
\xi_{\delta}(r)=\begin{cases}
-1/2,\quad&1-\delta<r<1\,,\\
a,\quad &1<r<1+\tilde\delta\,,\\
0,\quad &\text{otherwise}.
\end{cases}
\]
Here, we choose $a=a(\delta)$ from the condition
\begin{equation}\label{l0}
\int_0^{\infty}r\xi_{\delta}(r)\dr r=0,
\end{equation}
which is equivalent to
\begin{equation}\label{ln1}
a=\frac{\int_{1-\delta}^1r\dr r}{2\int_1^{1+\tilde\delta}r\dr r}\,.
\end{equation}
Obviously,
$a\to 0$ as $\delta\to 0$. Note also that for small $\delta$ we have
\begin{equation}\label{l1}
\int_0^{\infty}r\xi_{\delta}(r)J_0(j_{1,1}r)\dr r>0.
\end{equation}
Indeed, we obviously have
\[
\frac{d \int_0^{1}r\xi_{\delta}(r)J_0(j_{1,1}r)\dr r}{d\delta}=
-\frac{1}{2}\frac{d \int_{1-\delta}^{1}rJ_0(j_{1,1}r)\dr r}{d\delta}%=-\frac{1}{2}\bigl(rJ_0(j_{1,1}r)\bigr)\bigm|_{r=1}
=-\frac{J_0(j_{1,1})}{2},
\]
so
\begin{equation}\label{l2}
\int_0^{1}r\xi_{\delta}(r)J_0(j_{1,1}r)\dr r\sim \frac{-J_0(j_{1,1})\delta}{2}.
\end{equation}
Similarly, using \eqref{ln1} we obtain
\begin{equation}
\begin{split}
b:&=\frac{d \int_1^{\infty}r\xi_{\delta}(r)J_0(j_{1,1}r)\dr r}{d\delta}=
\frac{d (a\int_{1}^{1+\tilde\delta}rJ_0(j_{1,1}r)\dr r)}{d\delta}\\
&=\frac{\int_{1}^{1+\tilde\delta}rJ_0(j_{1,1}r)\dr r}{2\int_1^{1+\tilde\delta}r\dr r}
\frac{d \int_{1-\delta}^1r\dr r}{d\delta}
=\frac{\int_{1}^{1+\tilde\delta}rJ_0(j_{1,1}r)\dr r}{2\int_1^{1+\tilde\delta}r\dr r}\,,
\end{split}
\end{equation}
so
\begin{equation}\label{l3}
\int_1^{\infty}r\xi_{\delta}(r)J_0(j_{1,1}r)\dr r\sim b\delta
\end{equation}
as $\delta\to 0$. %where $b$ is an average value of the function $J_0(r)$ for $1\le r\le 1.1$.
Since
$j_{1,1}$ is a local minimum of $J_0$, we have $b>\frac{J_0(j_{1,1})}{2}$. Now formulas \eqref{l2} and \eqref{l3} imply \eqref{l1}.
We now fix a small $\delta<\tilde\delta$ for which \eqref{l1} holds and put $\tilde\zeta(r)=\zeta_0(r)+\xi_{\delta}(r)$.
Then, since
\begin{equation}\label{l4}
\int_0^{\infty}r\zeta_0(r)J_0(j_{1,1}r)\dr r=0,
\end{equation}
we have
\begin{equation}\label{l11}
\int_0^{\infty}r\tilde\zeta(r)J_0(j_{1,1}r)\dr r>0.
\end{equation}
It is easy to show that in fact for all positive $\gamma\le j_{1,1}$ we have
\begin{equation}\label{l5}
\int_0^{\infty}r\tilde\zeta(r)J_0(\gamma r)\dr r>0.
\end{equation}
Indeed, the function
\[
l(\gamma):=\int_0^{\infty}r\tilde\zeta(r)J_0(\gamma r)\dr r
\]
decreases for $\gamma<j_{1,1}$, since its derivative
\[
l'(\gamma)=-\int_0^{\infty}r^2\tilde\zeta(r)J_1(\gamma r)\dr r
\]
is negative as $J_1(r)$ is positive for $r\in [0,j_{1,1}]$.
Note also that \eqref{l0} implies
\begin{equation}\label{l6}
\int_0^{\infty}r\tilde\zeta(r)\dr r=1.
\end{equation}

Now let us construct a sequence of domains $\Omega_n$ which satisfy the following properties:
\begin{enumerate}
\item[(i)] the domain $\Omega_n$ is invariant under the rotation on $\frac{2\pi}{n}$ around the origin;
\item[(ii)] in the sector $-\frac{\pi}{n}\le\phi\le\frac{\pi}{n}$ in polar coordinates $(r,\phi)$
the domain $\Omega_n$ is given by $\{(r,\phi), %0\le r\le r_- \text{or}
\,|\phi|<\frac{\pi\tilde\zeta(r)}{n}\}$.
\end{enumerate}
Thus, for large $n$ the domain $\Omega_n$ has many thin spikes, see Figure \ref{fig:counterexample} for the picture of such a domain.

\begin{figure}[hbt!]
\begin{center}
\fbox{\includegraphics[width=0.6\textwidth]{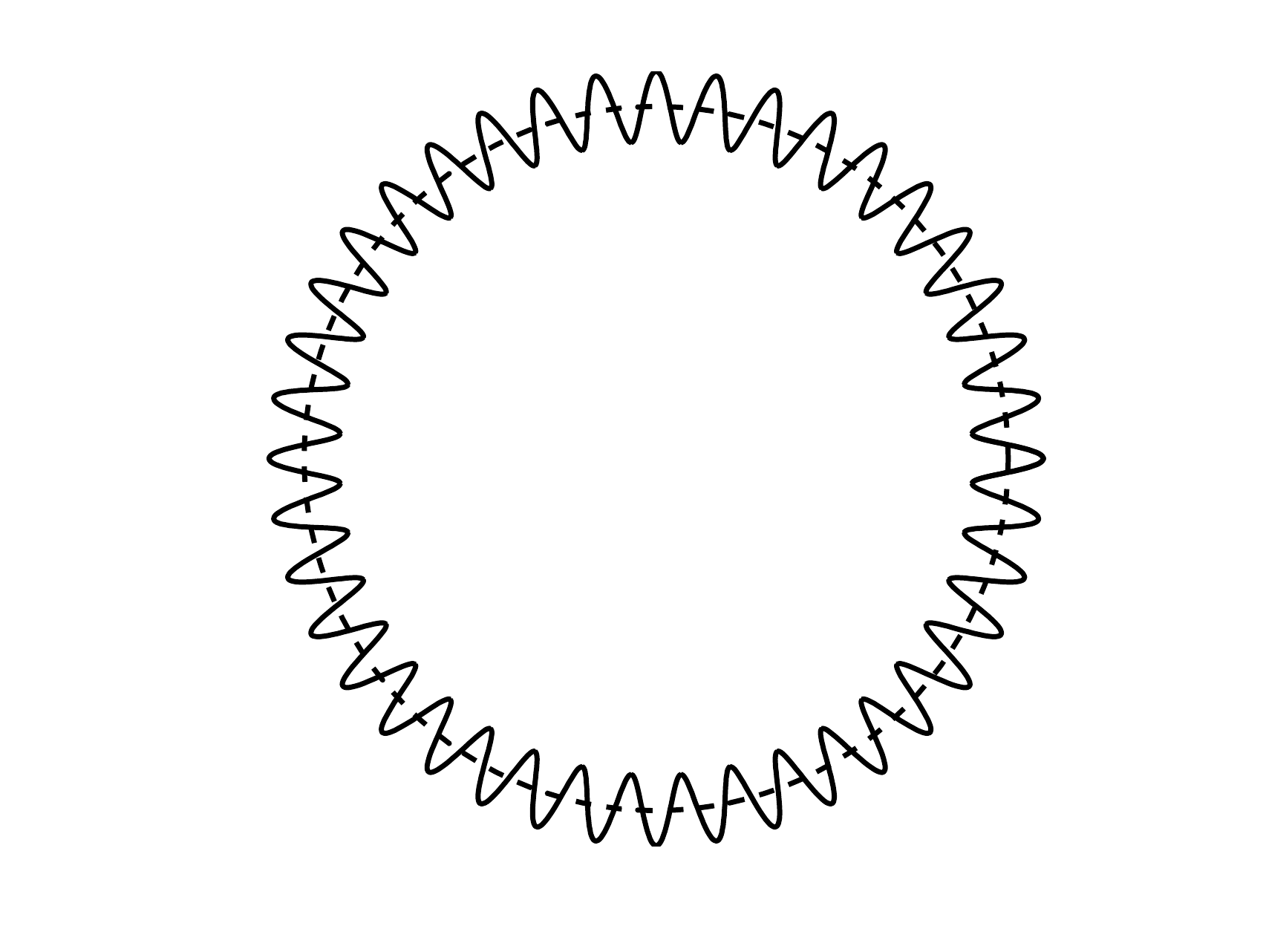}}
\end{center}
\caption{Domain $\Omega_n$.}\label{fig:counterexample}
\end{figure}

It is obvious that these properties determine the domain $\Omega_n$ uniquely and, moreover,
that %$\zeta$-function of each of these domains equals $\tilde\zeta$.
$\zeta_{\Omega_n}=\tilde\zeta$ for all $n$.
Note also that for
all positive $\gamma\le j_{1,1}$ and all unit vectors $\eb$ we have
\begin{equation}\label{l7}
\int_{\Omega_n}\cos(\gamma x_{\eb})\dr\xb\to\int_0^{\infty}r\tilde\zeta(r)J_0(\gamma r)\dr r
\end{equation}
as $n\to\infty$ uniformly over $\gamma$ and $\eb$. Therefore, \eqref{l5} implies that for sufficiently large $n$
\begin{equation}\label{l8}
\int_{\Omega_n}\cos(\gamma x_{\eb})\,\dr\xb>0
\end{equation}
for all positive $\gamma\le j_{1,1}$ and all unit vectors $\eb$. Thus, for this domain $\Omega_n$ we have
$\kappa(\Omega_n)>j_{1,1}$, finishing the proof of Theorem~\ref{thm:counterexample}.

It remains to prove Theorem \ref{Nazarov} and Corollary \ref{cor:Nazarov}. Suppose that we have proved Theorem \ref{Nazarov},
and thus constructed a sequence $I_n$ of one-dimensional balanced domains for which
$\kappa(I_n)\vol_1(I_n)\to\infty$ as $n\to\infty$. Consider
\begin{equation}
A_n:=(I_n)^2=\{\xb=(x_1,x_2)\,,x_1,x_2\in I_n\}.
\end{equation}
Then $\vol_2(A_n)=(\vol_1(I_n))^2$, and  $\kappa(A_n)=\kappa(I_n)$, and so $\kappa(A_n)\sqrt{\vol_2(A_n)}\to\infty$ as $n\to\infty$. Now it remains to connect
the disjoint rectangles in $A_n$ by narrow corridors to construct
connected domains $ \tilde{A}_n$ with $\kappa(\tilde A_n)\sqrt{\vol_2(\tilde A_n)}\to\infty$ as $n\to\infty$, proving Corollary \ref{cor:Nazarov}.

Let us prove Theorem \ref{Nazarov}. We formulate the following
\begin{lem}\label{L2}
For each positive $C$ there exist a natural number $n$ and real numbers $w_1,\dots,w_n$ such that $w_1\ge 1$, $w_{j+1}\ge w_j+1$,
and the function $f(\xi):=\sum_{j=1}^n\cos(w_j\xi)$ is positive for $\xi\in [-C/n,C/n]$.
\end{lem}
\begin{proof}
The proof is due to F. Nazarov \cite{Naz}. Put, for real $t$, $g(t):=(1-|t|)^{2}_+$. Then the Fourier transform $\widehat{g}(s)$ of $g$ is positive for
real $s$. Denote, for real $x$,
\begin{equation}
G(x):=1+2\sum_{k=1}^n\left(1-\frac{k}{n}\right)^2\cos (kx)
=\sum_{k\in\Z}g\left(\frac{k}{n}\right)\er^{\ir kx}.
\end{equation}
Then the Poisson summation formula implies that
\begin{equation}
G(x)= n\sum_{m\in\Z}\widehat{g}(n(x+2\pi m))\ge n \widehat{g}(nx),
\end{equation}
and so $G(x)\ge cn$ whenever $|x|\le C/n$. Now put
\begin{equation}
F(x):=\sum_{k=1}^n a_k^n\cos (kx),
\end{equation}
where $a_k^n$ is a collection of independent random variables such that $a_k^n=1$ with probability $\left(1-\frac{k}{n}\right)^2$; otherwise
$a_k^n=0$. Then the standard probabilistic arguments based on the large deviation principle imply that
for each fixed point $x$ the probability of the event
\begin{equation}\label{L1}
|F(x)-(G(x)-1)/2|\ge n^{3/4}
\end{equation}
is $O(e^{-n^{1/4}})$. In particular, putting $x=0$ in \eqref{L1}, we see that the number of coefficients $a^n_k$ which are equal to one, is at least $n/10$ with probability $1-O(n^3e^{-n^{1/4}})$.
Put $x_j:=\frac{j}{n^3}$, $j=0,\dots,n^3$. Then the probability of the
event that for all $j=0,\dots,n^3$ we have
\begin{equation}\label{L1b}
|F(x_j)-(G(x_j)-1)/2|\le n^{3/4}
\end{equation}
is at least $1-O(n^3e^{-n^{1/4}})$ and thus is positive for sufficiently large $n$. Since the derivative
of both $F$ and $G$ is $O(n^2)$, this means that the probability that \eqref{L1} is satisfied for all
$x\in [0,1]$ is positive when $n$ is large. Therefore, for each large $n$ there is at least one $F$ such that \eqref{L1} is satisfied for all $x\in [0,1]$. Thus chosen $F$ satisfies $F(x)\ge cn/2$ for $|x|\le C/n$ .
\end{proof}

To finish the proof of Theorem \ref{Nazarov}, we now take, for a given $C>0$, the numbers $n$  and  $w_j$, $j=1,\dots,n$,  from Lemma \ref{L2}, and define $I_n:=\{x\in\Rbb,\, ||x|-w_j|\le 1/2\text{ for some } j\}$. Then $\vol_1(I_n)\le 2n$, and
\[
\widehat{\chi_{I_n}}(\xi)=\frac{4\sin(\xi/2)}{\xi}\sum_{j=1}^n \cos(w_j \xi)\,.
\]
Therefore by Lemma \ref{L2}  for any constant $C$ we have $\kappa(I_n)\ge C/n$ for sufficiently large $n$. 

%%%%%%%%%%%%%%%%%%%%%%%%%%%%%%%%%%%%%%%%%%%%%%%%%%%%%%%%%%%%

\end{document}